\documentclass[12pt,thmsa,sw20lart]{article}
\usepackage{amsmath,amsthm,amssymb,epsfig,graphics,subfigure, wrapfig, epsfig, color}
\usepackage{tikz}
\input amssym.def
\input amssym.tex
\date{}
\newtheorem{lemma}{Lemma}

\newtheorem{theorem}{THEOREM}

\theoremstyle{definition}
\newtheorem{example}{Example}
\newtheorem{remark}{Remark}
\newtheorem*{acknowledgement}{Acknowledgements}

\allowdisplaybreaks

\begin{document}

\title{{\bf Boundedness of the Riesz potential in central Morrey--Orlicz spaces}}
\author {Evgeniya Burtseva, Lech Maligranda and Katsuo Matsuoka}

\maketitle


\begin{abstract}
\noindent {\footnotesize Boundedness of the maximal function and the Calder\'on--Zygmund singular integrals 
in central Morrey--Orlicz spaces were proved in papers by the second and third authors. The weak-type 
estimates have also been proven. Here we show boundedness of the Riesz potential in central 
Morrey--Orlicz spaces and the corresponding weak-type version.}
\end{abstract}

\section{Orlicz spaces and central Morrey--Orlicz spaces} \label{sec:central}

First of all, we recall the definition of Orlicz spaces on $\mathbb R^n$ and some of their properties to be used later on 
(see \cite{KR61} and \cite{Ma89} for details). 

A function $\Phi \colon [0,\infty) \to [0,\infty)$ is called an {\it Orlicz function}, if it is an increasing continuous 
and convex function with $\Phi(0)=0$. Each such a function $\Phi$ has an integral representation 
$\Phi(u) = \int_0^u \Phi^{\prime}_{+}(t) \,dt$, where the right-derivative $\Phi^{\prime}_{+}$ is a nondecreasing 
right-continuous function (see \cite[Theorem 1.1]{KR61}). We will write below estimates for everywhere 
differentiable Orlicz function $\Phi$, but then using the above integral representation, these estimates will be 
true for almost all $u > 0$ with its right-derivative $\Phi^{\prime}_+$ instead of derivative $\Phi^{\prime}$. 
Of course, we have estimates
\begin{equation} \label{PP1}
\Phi(u) \leq u\, \Phi^{\prime}(u) \leq \Phi(2u)  ~~ {\rm for ~ all} ~~ u > 0.
\end{equation} 

If we want to include in the Orlicz spaces, for example, spaces $L^{\infty}({\mathbb R^n}), L^p ({\mathbb R^n}) 
\cap L^{\infty}({\mathbb R^n})$ and $L^p ({\mathbb R^n}) + L^{\infty}({\mathbb R^n})$ for $1 \leq p < \infty$, 
then we need to consider a broader class of functions than Orlicz's functions, the so-called Young functions. 
A function $\Phi \colon [0,\infty) \to [0,\infty]$ is called a {\it Young function}, if it is a nondecreasing convex 
function with $\lim_{u \to 0^+}\Phi(u) = \Phi(0) = 0$, and not identically $0$ or $\infty$ in $(0,\infty)$. 
It may have jump up to $\infty$ at some point $u>0$, but then it should be left continuous at $u$. 

Let $(\Omega, \Sigma, \mu)$ be a $\sigma$-finite complete nonatomic measure space and 
$L^0(\Omega)$ be the space of all $\mu$-equivalent classes of real-valued and $\Sigma$-measurable functions defined 
on $\Omega$. 

For any Young function $\Phi$, the {\it Orlicz space} $L^{\Phi}(\Omega)$, which contains all $f \in L^0(\Omega)$ 
such that $\int_{\Omega} \Phi( \varepsilon |f(x)|) \, d\mu(x) < \infty$ for some 
$ \varepsilon = \varepsilon(f) > 0$ with the {\it Luxemburg--Nakano norm} 
\begin{equation} \label{LNnorm}
\|f\|_{L^{\Phi}} = \inf \left\{\varepsilon > 0 \colon \int_{\Omega} \Phi \big (\frac {|f(x)|}{\varepsilon} \big) 
\, d\mu(x) \leq 1 \right\},  
\end{equation}
is a Banach space (cf. \cite[pp. 70--71]{KR61}, \cite[pp. 15--16]{Ma89}, \cite[pp. 125--127]{Ma11} and \cite[pp. 67--68]{RR91}). 
The {\it fundamental function} of the Orlicz space $L^{\Phi}(\Omega)$ is 
$$
\varphi_{L^{\Phi}(\Omega)}(t) = \|\chi_A \|_{L^{\Phi}(\Omega)} =  \|\chi_{[0, \mu(A)]} \|_{L^{\Phi}([0, \infty))} = 1/{\Phi^{-1}(1/t)},
$$
where $\chi_{A}$ is the characteristic function of the set $A \subset \Omega, t = \mu(A)$ and $\Phi^{-1}$ is the right-continuous
inverse of $\Phi$ defined by $\Phi^{-1}(v) = \inf \, \{u \ge 0 \colon \Phi(u) > v\}$ with $\inf \, \emptyset = \infty$. 

To each Young function $\Phi$ one can associate another convex function $\Phi^*$, i.e., the {\it complementary function} to 
$\Phi$, which is defined by 
$$
\Phi^* (v) = \sup_{u>0} \, [uv - \Phi(u)] ~~ {\rm for} ~~ v \geq 0. 
$$
Then $\Phi^*$ is also a Young function and $\Phi^{**} = \Phi$. Note that $u \leq \Phi^{-1}(u) \Phi^{*^{-1}}(u) \leq 2 u$ for all $u > 0$. 

We say that a Young function $\Phi $ satisfies the $\Delta_2$-{\it condition} and we write shortly $\Phi \in \Delta_2$, if 
$0< \Phi(u) < \infty$ for $u>0$ and there exists a constant $D_2 \geq 1$ such that 
\begin{equation} \label{D2}
\Phi(2u) \leq D_2 \Phi(u) ~~ {\rm for ~all} ~~ \ u > 0. 
\end{equation}

In this paper we consider Orlicz spaces $L^{\Phi}(\mathbb R^n)$ on $\mathbb R^n$ with the Lebesgue measure.
Then we define the Morrey--Orlicz spaces $M^{\Phi, \lambda}({\mathbb R^n})$ and central Morrey--Orlicz spaces 
$M^{\Phi, \lambda}(0)$. In the 2000s, several authors (for example, F. Deringoz, V. S. Guliyev, J. J. Hasanov, 
T.~Mizuhara, E. Nakai, S. Samko, Y. Sawano, H. Tanaka and others) defined Orlicz versions of the Morrey space, 
i.e., Morrey--Orlicz spaces, and investigated the boundedness for the Hardy--Littlewood maximal operator and other 
operators on them (see, for example, \cite{DGNSS19}, \cite{GD14}, \cite{GDH17}, \cite{Na04}, \cite{SST12} and 
the references therein). The Orlicz version of central Morrey spaces, i.e., central Morrey--Orlicz spaces were defined 
in papers by the second and third authors. They investigated boundedness on central Morrey--Orlicz spaces of the 
Hardy--Littlewood maximal operator in paper \cite{MM15} and also boundedness of the Calder\'on--Zygmund singular 
integrals on them in paper \cite{MM20}. In this paper we present conditions under which the Riesz potential
is bounded on central Morrey--Orlicz spaces.

For any Young function $\Phi$, number $\lambda \in \mathbb R$, a set $A \subset \mathbb R^n$ with $0 < |A| < \infty$ 
and for $f \in L^0({\mathbb R^n})$ let
$$
\|f\|_{\Phi, \lambda, A} = \inf \big\{ \varepsilon > 0\colon \dfrac 1 {|A|^{\lambda}} \int_{A} \Phi\big(\dfrac {|f(x)|} \varepsilon \big) 
\,dx \leq 1 \big\}, 
$$
and the corresponding (smaller) expression 
$$
\|f\|_{\Phi, \lambda, A, \infty} = \inf \big\{ \varepsilon > 0 \colon \sup_{u>0} \Phi(\dfrac{u}{\varepsilon}) \, \dfrac 1 {|A|^{\lambda}} \, d(f\chi_{A}, u) 
\leq 1\big\}, 
$$
where $d(f, u) = | \{x \in {\mathbb R^n} \colon |f(x)| > u \}|$. Note that $\|f\|_{\Phi, \lambda, A, \infty} \leq \|f\|_{\Phi, \lambda, A} $
provided that the expression on the right is finite. In fact, if $\|f\|_{\Phi, \lambda, A} < c$, then for arbitrary $u > 0$ we have
$$
1 \geq \dfrac{1}{|A|^{\lambda}} \int_{A} \Phi\big(\dfrac {|f(x)|}{c} \big) \,dx \geq 
\dfrac{1}{|A|^{\lambda}} \int_{\{x \in A \colon |f(x)| > u\}} \Phi\big(\dfrac {|f(x)|}{c} \big) \,dx 
\geq \dfrac{1}{|A|^{\lambda}} \,\Phi(\dfrac{u}{c}) \, d(f\chi_{A}, u),
$$
and $\|f\|_{\Phi, \lambda, A, \infty} \leq c$. Hence, $\|f\|_{\Phi, \lambda, A, \infty} \leq \|f\|_{\Phi, \lambda, A}$.

Using these notions and considering open balls $B(x_0, r)$ with a center at $x_0 \in \mathbb R^n$ and radius $r > 0$, i.e.
$B(x_0, r) = \{x \in {\mathbb R^n} \colon |x - x_0| < r\}$, and also open balls $B(0, r) = B_r$ with a center at $0$ we can define
{\it Morrey--Orlicz spaces} $M^{\Phi, \lambda}({\mathbb R^n})$ and {\it weak Morrey--Orlicz spaces} 
$WM^{\Phi, \lambda}({\mathbb R^n})$:
\begin{equation} \label{MOnorm}
M^{\Phi, \lambda}({\mathbb R^n}) = \left\{ f \in L^1_{loc}({\mathbb R^n})\colon \|f\|_{M^{\Phi, \lambda}} = 
\sup_{x_0 \in {\mathbb R^n}, r > 0} \|f\|_{\Phi, \lambda, B(x_0, r)} < \infty \right\} 
\end{equation}
and
\begin{equation} \label{wMOnorm}
WM^{\Phi,\lambda}({\mathbb R^n}) = \left\{ f \in L^1_{loc}({\mathbb R^n})\colon \|f\|_{WM^{\Phi,\lambda}} = 
\sup_{x_0 \in {\mathbb R^n}, r > 0} \|f\|_{\Phi, \lambda, B(x_0, r), \infty} < \infty \right\}. 
\end{equation}
Similarly, we can define {\it central Morrey--Orlicz spaces} $M^{\Phi, \lambda}(0)$ and {\it weak central Morrey--Orlicz spaces} 
$WM^{\Phi, \lambda}(0)$:
\begin{equation} \label{CMOnorm}
M^{\Phi, \lambda}(0) = \left\{ f \in L^1_{loc}({\mathbb R^n})\colon \|f\|_{M^{\Phi, \lambda}(0)} = \sup_{r > 0} \|f\|_{\Phi, \lambda, B_r} 
< \infty \right\} 
\end{equation}
and
\begin{equation} \label{wCMOnorm}
WM^{\Phi,\lambda}(0) = \left\{ f \in L^1_{loc}({\mathbb R^n})\colon \|f\|_{WM^{\Phi,\lambda}(0)} = 
\sup_{r > 0} \|f\|_{\Phi, \lambda, B_r, \infty} < \infty \right\}. 
\end{equation}

All these spaces are Banach ideal spaces on ${\mathbb R^n}$ (sometimes they are $\{0\}$, that is, they contain 
only all functions equivalent to $0$ on $\mathbb R^n$). Moreover, we have continuous embeddings 
$M^{\Phi, \lambda}({\mathbb R^n}) \overset{1}{\hookrightarrow} WM^{\Phi, \lambda}({\mathbb R^n}), 
M^{\Phi, \lambda}(0) \overset{1}{\hookrightarrow} WM^{\Phi, \lambda}(0)$ and also 
$M^{\Phi, \lambda}({\mathbb R^n}) \overset{1}{\hookrightarrow} M^{\Phi, \lambda}(0), \newline
WM^{\Phi, \lambda}({\mathbb R^n})\overset{1}{\hookrightarrow} WM^{\Phi, \lambda}(0)$.

Let us recall that the normed subspace $X = (X, \| \cdot\|_X)$ of $L^0(\Omega)$ is an {\it ideal space} 
on $\Omega$: if $f, g \in X$ with $|f(x)| \leq |g(x)|$ for $\mu$-almost all $x \in \Omega$, and $g \in X$, then 
$f \in X$ and $\| f\|_X \leq \| g\|_X$. 
Here and further, for two Banach ideal spaces $X$ and $Y$, we use the symbol $X\hookrightarrow Y$ rather than 
$X\subset Y$ for continuous embedding. Moreover, the symbol $X\overset{C}{\hookrightarrow }Y$ indicates that 
$X\hookrightarrow Y$ with the norm of the embedding operator not bigger than $C$, i.e., $\| f\|_{Y} \leq C\, \|f\|_{X}$ 
for all $f\in X$. 

Note that Morrey--Orlicz spaces and central Morrey--Orlicz spaces are generalizations of Orlicz spaces and Morrey 
spaces (on ${\mathbb R^n}$). In particular, we can obtain the following spaces (see \cite{MM15} for more details): 
\begin{enumerate}
\item[$({\rm i})$] (Orlicz and weak Orlicz spaces) If $\lambda = 0$, then 
$$
M^{\Phi, 0}({\mathbb R^n}) = M^{\Phi, 0}(0) = L^{\Phi}({\mathbb R^n}) \ \ {\text{\rm{and}}} \ \ 
   WM^{\Phi, 0}({\mathbb R^n}) = WM^{\Phi, 0}(0) = WL^{\Phi}({\mathbb R^n}).
$$
\item[$({\rm ii})$] (Beurling--Orlicz and weak Beurling--Orlicz spaces) If $\lambda = 1$, then 
$$
M^{\Phi, 1}({\mathbb R^n}) = B^{\Phi}({\mathbb R^n}) \ \ {\text{\rm{and}}} \ \ 
               WM^{\Phi, 1}({\mathbb R^n}) = WB^{\Phi}({\mathbb R^n}). 
$$
As for $B^{\Phi}({\mathbb R^n})$ and $WB^{\Phi}({\mathbb R^n})$, see \cite{MM15}.

\item[$({\rm iii})$] (classical Morrey, weak Morrey, central Morrey and weak central Morrey spaces) If $\Phi(u) = u^p, 1 \le p< \infty$ 
and $\lambda \in \mathbb R$, then $M^{\Phi, \lambda}({\mathbb R^n}) = M^{p, \lambda}({\mathbb R^n}), 
WM^{\Phi,\lambda}({\mathbb R^n}) = WM^{p,\lambda}({\mathbb R^n})$ and $M^{\Phi, \lambda}(0) = 
M^{p, \lambda}(0), WM^{\Phi,\lambda}(0) = WM^{p,\lambda}(0)$.
\newline Here $M^{p, \lambda}({\mathbb R^n})$, $WM^{p, \lambda}({\mathbb R^n})$, $M^{p, \lambda}(0)$, 
$WM^{p, \lambda}(0)$ are the classical Morrey, weak Morrey, central Morrey and weak central Morrey spaces, 
respectively. 
\end{enumerate}
We want to note that $M^{p, \lambda}({\mathbb R^n}) \neq \{0\}$ if and only if $0 \leq \lambda \leq 1$ (see \cite[Lemma 1]{BG04}) 
and $M^{p, \lambda}(0) \neq \{0\}$ if and only if $\lambda \geq 0$ (see \cite{Bu12}, \cite{BG04}, \cite{BJT}). Moreover, 
$M^{p, 0}({\mathbb R^n}) = M^{p, 0}(0) = L^p({\mathbb R^n})$ and $M^{p, 1}({\mathbb R^n}) = L^{\infty}({\mathbb R^n})$ (see 
\cite[Theorem 4.3.6]{KJF77}). However, $ L^{\infty}({\mathbb R^n})  \overset{1}{\hookrightarrow} M^{p, 1}(0)$ and the inclusion 
is strict. For example, in one-dimensional case $f(x) = \sum\limits_{n=0}^{\infty}2^{n/p} \chi_{[n, n+2^{-n}]}(|x|) \in M^{p, 1}(0)
 \setminus L^{\infty}({\mathbb R^1})$. Of course, for $0 \leq \lambda \leq 1$ the inclusion 
 $M^{p, \lambda}({\mathbb R^n}) \overset{1}{\hookrightarrow} M^{p, \lambda}(0)$ holds and is strict for $0 < \lambda \leq 1$ 
 (a suitable example we can find in \cite[p. 156]{GP18}). It is also true that if $1 \leq p < q < \infty, 0 \leq \mu < \lambda < 1$ and 
 $\frac{1 - \lambda}{p} = \frac{1 - \mu}{q}$, then 
 \begin{equation} \label{embeddings}
 M^{q, \mu}({\mathbb R^n}) \overset{1}{\hookrightarrow} M^{p, \lambda}({\mathbb R^n})  
 \quad {\rm and} \quad M^{q, \mu}(0) \overset{1}{\hookrightarrow} M^{p, \lambda}(0).
 \end{equation}
 Both inclusions are proper (see, for example, \cite{GHI18}); the second embedding in (\ref{embeddings}) is also true for $1 < \lambda < \mu$.
The embeddings (\ref{embeddings}) follow by the H\"older--Rogers inequality with $\frac{q}{p} > 1$, since for any 
$x_0 \in \mathbb R^n$ we have
\begin{eqnarray*}
\int_{B(x_0, r)} |f(x)|^p \, dx 
&\leq& 
\left(\int_{B(x_0, r)} |f(x)|^q \, dx \right)^{p/q} |B(x_0, r)|^{1 - p/q} \\
&=& 
\left(\frac{1}{|B(x_0, r)|^{\mu}} \int_{B(x_0, r)} |f(x)|^q \, dx \right)^{p/q} |B(x_0, r)|^{1 - p/q + \mu p/q}\\
&=& 
\left(\frac{1}{|B(x_0, r)|^{\mu}} \int_{B(x_0, r)} |f(x)|^q \, dx \right)^{p/q} |B(x_0, r)|^{\lambda},
\end{eqnarray*}
and from the fact that $1 - p/q + \mu p/q = (\mu - 1) \, p/q + 1 = - (1 - \lambda) + 1 = \lambda$.

\vspace{2mm}

If the supremum in definitions (\ref{MOnorm})--(\ref{wCMOnorm}) is taken over all $r > 1$, then we will have 
corresponding definitions of non-homogeneous Morrey--Orlicz spaces, non-homogeneous weak Morrey--Orlicz spaces, 
non-homogeneous central Morrey--Orlicz spaces and non-homogeneous weak central Morrey--Orlicz spaces.

\section{The Riesz potential in Lebesgue, Orlicz and Morrey spaces} \label{sec:Rp}

The {\it Riesz potential} of order $\alpha \in (0, n)$ of a locally integrable function $f \in {\mathbb R^n}, n \geq 1$, is defined as
\begin{equation} \label{Riesz}
I_{\alpha}f(x)=\int_{\mathbb R^n}\dfrac {f(y)} {|x-y|^{n-\alpha}} \,dy, \quad {\rm for} \, \,\, x \in \mathbb R^n.
\end{equation}
The linear operator $I_{\alpha}$ plays a role in various branches of analysis, including potential theory, harmonic analysis, 
Sobolev spaces and partial differential equations. Therefore, investigations of the boundedness of the operator $I_{\alpha}$ 
between different spaces are important.

The classical {\it Hardy--Littlewood--Sobolev theorem} states that if  $1 < p < q < \infty$, then a Riesz potential 
$I_{\alpha}$ is of strong-type $(p,q)$, that is, bounded from $L^p({\mathbb R^n})$ to $L^q({\mathbb R^n})$ 
if and only if $1/q=1/p-\alpha/n$. For $p = 1 < q < \infty$ Zygmund proved that $I_{\alpha}$ is of weak-type $(1, q)$, 
that is, bounded from $L^1({\mathbb R^n})$ to $WL^q({\mathbb R^n})$, where $1/q=1-\alpha/n$. 
The weak-$L^q$ space $WL^q({\mathbb R^n}) = L^{q, \infty}({\mathbb R^n})$, called also the Marcinkiewicz space, 
consists of all $f \in L^0({\mathbb R^n})$ such that the quasi-norm 
$\| f\| _{q, \infty} = \sup_{t > 0} t \, | \{x \in {\mathbb R^n} \colon |f(x)| > t \}|^{1/q}$ is finite.
The proofs of these results we can find in the books \cite[pp. 125--127]{Ga07}, \cite[pp. 2--5]{Gr09}, \cite[pp. 117--121]{St70}, 
\cite[pp. 150--154]{To86} and \cite[pp. 86--87]{Zi89}. 

The boundedness of $I_{\alpha}$ from an Orlicz space $L^{\Phi}(\mathbb R^n)$ to another Orlicz space $L^{\Psi}(\mathbb R^n)$ 
was studied by Simonenko (1964), O'Neil (1965) and Torchinsky (1976) under some restrictions on the Orlicz functions $\Phi$ 
and $\Psi$. In 1999 Cianchi \cite{Ci99} gave a necessary and sufficient condition for the boundedness of $I_{\alpha}$ 
from $L^{\Phi}(\mathbb R^n)$ to $L^{\Psi}(\mathbb R^n)$ and from $L^{\Phi}(\mathbb R^n)$ to weak Orlicz space 
$WL^{\Psi}(\mathbb R^n)$. Another sufficient conditions for boundedness of the Riesz operator $I_{\alpha}$ (and even 
for a generalized fractional operator $I_{\rho}$) were given in 2001 by Nakai \cite{Na01a}, \cite{Na01b}. Then in 2017, 
Guliyev--Deringoz--Hasanov in \cite[Theorem 3.3]{GDH17} gave more readable necessary and sufficient conditions 
for the boundedness of $I_{\alpha}$ from $L^{\Phi}(\mathbb R^n)$ to $WL^{\Psi}(\mathbb R^n)$ and from 
$L^{\Phi}(\mathbb R^n)$ to $L^{\Psi}(\mathbb R^n)$. 
 
Results concerning boundedness of the Riesz potential between Morrey spaces were first obtained by Spanne with 
the Sobolev exponent $1/q=1/p-\alpha/n$, and this result was published in 1969 by Peetre \cite{Pe69}: if 
$0 < \alpha < n, 1 < p < n(1 - \lambda)/\alpha, 0 < \lambda < 1, 1/q=1/p-\alpha/n$ and 
$\lambda/p = \mu/q$, then the Riesz potential $I_{\alpha}$ is bounded from 
$M^{p, \lambda}(\mathbb R^n)$ to $M^{q, \mu}(\mathbb R^n)$. Then in 1975 a stronger result was obtained by 
Adams \cite{Ad75}, and reproved by Chiarenza--Frasca \cite{CF87}. Adams proved boundedness of $I_{\alpha}$ from 
$M^{p, \lambda}(\mathbb R^n)$ to $M^{q_1, \lambda}(\mathbb R^n)$ with a better exponent $q_1$, namely $
1/q_1 = 1/p - \alpha/[n(1-\lambda)]$. Adams result is stronger than the Peetre--Spanne theorem because 
$q < q_1$ and $(1 - \mu)/q = (1 - \lambda)/q_1$, from which follows the embedding $M^{q_1, \lambda}(\mathbb R^n) 
\overset{1}{\hookrightarrow } M^{q, \mu}(\mathbb R^n)$ and this means that the target space 
$M^{q_1, \lambda}(\mathbb R^n)$ is smaller than target space $M^{q, \mu}(\mathbb R^n)$ in the Peetre--Spanne 
result.

Central Morrey spaces $M^{p, \lambda}(0)$ were first introduced in \cite[p. 607]{GH94} and in \cite[p. 5]{AGL00} (see 
also \cite[p. 257]{CL89} and \cite[p. 500]{GC89} for $\lambda = 1$). Result on the boundedness of the Riesz potential 
in these spaces was proved by Fu--Lin--Lu \cite[Proposition 1.1]{FLL08}: if 
$1 < p < n (1 - \lambda)/\alpha, 0 < \lambda < 1, 1/q = 1/p - \alpha/n$ 
and $\lambda/p = \mu/q$, then the Riesz potential $I_{\alpha}$ is bounded from $M^{p, \lambda}(0)$ to $M^{q, \mu}(0)$ 
(see also \cite{BGGM11}, where the result is proved even for more general local Morrey-type spaces). Komori-Furuya 
and Sato \cite[Proposition 1]{KS17} showed that Adams type result on boundedness in central Morrey spaces does 
not hold. They showed that if $\frac{1 - \mu}{q} = \frac{1 - \lambda}{p} - \frac{\alpha}{n}$ and $ \alpha/n < 1/p - 1/q < \alpha/[n(1 - \lambda)]$, 
then $I_{\alpha}$ is not bounded from $M^{p, \lambda}(0)$ to $M^{q, \mu}(0)$ because 
$\mu/q = \lambda/p - (1/p - \alpha/n - 1/q) < \lambda/p$.  

We will generalize the last results to central Morrey--Orlicz spaces. In Theorem 2, the necessary conditions for boundedness 
of $I_{\alpha}$ are given, and in Theorem 3 -- sufficient conditions are presented.

In the proof of boundedness of the Riesz potential in the central Morrey-Orlicz spaces we will need some necessary 
estimates. We will present them in the next section.
 
\section{Some technical results}

To prove the main results of this paper, we need some technical calculations. In order not to hide the main 
ideas in proofs of the main results we collect such calculations in Lemma~\ref{lemma1} below.

\begin{lemma}\label{lemma1} Let $\Phi$ be a Young function, $\Phi^*$ its complementary function, $0 \le \lambda \le 1$ and $r > 0$. 
Then
\begin{itemize} 
\item[$({\rm i})$] $\int_{B_r} |f(x)g(x)|\,dx \leq 2\, |B_r|^{\lambda}\, \|f\|_{\Phi, \lambda, B_r} \|g\|_{\Phi^*, \lambda, B_r}$.
\item[$({\rm ii})$] $\| \chi_{B(x_0, r_0)}\|_{\Phi^{\ast}, \lambda, B_r} \le \frac{|B_r \cap B(x_0, r_0)|}{|B_r|^\lambda}\Phi^{-1} 
\left(\frac{|B_r|^\lambda}{|B_r \cap B(x_0, r_0)|} \right) $, {\it where $B_r \cap B(x_0, r_0) \neq \emptyset$ for $x_0 \in \mathbb 
R^n$ and $r_0 > 0$.}

In particular, $\|\chi_{B_r} \|_{\Phi^*, \lambda, B_r} \leq \dfrac {\Phi^{-1} \left( |B_r|^{\lambda-1} \right)} {|B_r|^{\lambda-1}}.$
\item[$({\rm iii})$] $\| \chi_{B_t}\|_{\Phi, \lambda, B_r} = 1/ {\Phi^{-1} \left(\frac{|B_r|^\lambda}{|B_r \cap B_t|}\right) }$ and 
$\|\chi_{B_t}\|_{M^{\Phi, \lambda} (0)} = \dfrac{1}{\Phi^{-1}(|B_t|^{\lambda - 1})}$ {\it for any $t > 0$.}
\item[$({\rm iv})$] $\| \chi_{B_t}\|_{\Phi, \lambda, B_r, \infty} = 1/ {\Phi^{-1} \left(\frac{|B_r|^\lambda}{|B_r \cap B_t|}\right) }$ and 
$\|\chi_{B_t}\|_{WM^{\Phi, \lambda} (0)} = \dfrac{1}{\Phi^{-1}(|B_t|^{\lambda - 1})}$ {\it for any $t > 0$.}
\end{itemize}
\end{lemma}
\begin{proof} (i) This estimate was proved in \cite[Lemma 2.6]{MM20}. 

(ii) Since for $u > 0$ we have $\Phi^{\ast} \left(\frac{u}{\Phi^{-1}(u)} \right) \le u$ (cf. Lemma 2.6 in \cite{MM20}) it 
follows for $u = \frac{|B_r|^\lambda}{|B_r \cap B(x_0, r_0)|}$ that 
$$
\int_{B_r} \Phi^{\ast} (\frac{\chi_{B(x_0, r_0)}(x) |B_r|^\lambda}{\Phi^{-1}
\left(\frac{|B_r|^\lambda}{|B_r \cap B(x_0, r_0)|} \right) |B_r \cap B(x_0, r_0)|} ) \, dx
$$
$$
= \int_{B_r \cap B(x_0, r_0)} \Phi^{\ast} (\frac{|B_r|^\lambda}{\Phi^{-1} (\frac{|B_r|^\lambda}{|B_r \cap B(x_0, r_0)|}) 
|B_r \cap B(x_0, r_0)|}) \, dx
$$
$$ \le \frac{|B_r|^\lambda}{|B_r \cap B(x_0, r_0)|}  \int_{B_r \cap B(x_0, r_0)} dx = |B_r|^\lambda.
$$
Hence, $\| \chi_{B(x_0, r_0)}\|_{\Phi^{\ast}, \lambda, B_r} \le \Phi^{-1} \left(\frac{|B_r|^\lambda}{|B_r \cap B(x_0, r_0)|} \right) 
\frac{|B_r \cap B(x_0, r_0)|}{|B_r|^\lambda}$, 
and (ii) follows.

(iii) Let $t > 0$. Since $\Phi (\Phi^{-1} (u)) \leq u$ for any $u > 0$  it follows that
$$
\int\limits_{B_r} \Phi \left(\chi_{B_t}(x) \, \Phi^{-1} (\frac{|B_r|^\lambda}{|B_r \cap B_t|}) \right) dx 
= \int\limits_{B_r \cap B_t} \Phi \left( \Phi^{-1} (\frac{|B_r|^\lambda}{|B_r \cap B_t|}) \right) dx 
$$
$$
 \leq \int\limits_{B_r \cap B_t} \frac{|B_r|^\lambda \, }{|B_r \cap B_t|} dx = |B_r|^\lambda,
$$
and so $\| \chi_{B_t}\|_{\Phi, \lambda, B_r} \leq 1/ {\Phi^{-1} \left(\frac{|B_r|^\lambda}{|B_r \cap B_t|}\right) }$. 
On the other hand, 
$$
1 \geq \frac{1}{|B_r|^\lambda} \int\limits_{B_r} \Phi \left(\frac{\chi_{B_t}(x)}{\| \chi_{B_t}\|_{\Phi, \lambda, B_r}} \right) dx 
= \Phi \left(\frac{1}{\| \chi_{B_t}\|_{\Phi, \lambda, B_r}} \right) \frac{|B_r \cap B_t|}{|B_r|^\lambda},
$$
or
$$
\frac{|B_r|^\lambda}{|B_r \cap B_t|} \geq \Phi (\frac{1}{\| \chi_{B_t}\|_{\Phi, \lambda, B_r}} ). 
$$
Since $u \leq \Phi^{-1} (\Phi (u))$ for any $u > 0$ such that $\Phi (u) < \infty$ we obtain
$$
\Phi^{-1} \left(\frac{|B_r|^\lambda}{|B_r \cap B_t|} \right)\geq \Phi^{-1} \left(\Phi (\frac{1}{\| \chi_{B_t}\|_{\Phi, \lambda, B_r}}) \right) 
\geq \frac{1}{\| \chi_{B_t}\|_{\Phi, \lambda, B_r}}, 
$$
which together with the previous estimate gives equality 
$\| \chi_{B_t}\|_{\Phi, \lambda, B_r} = 1/ {\Phi^{-1} \left(\frac{|B_r|^\lambda}{|B_r \cap B_t|}\right) }$.
Thus,
\begin{eqnarray*}
\|\chi_{B_t} \|_{M^{\Phi, \lambda}(0)} 
&=& 
\sup\limits_{r>0} \| \chi_{B_t}\|_{\Phi, \lambda, B_r} 
= \sup\limits_{r>0} \frac{1}{ {\Phi^{-1} \left(\frac{|B_r|^\lambda}{|B_r \cap B_t|}\right) }} \\
&=& \max \Big[ \sup\limits_{r\le t} \frac{1}{\Phi^{-1} \left(\frac{|B_r|^\lambda}{|B_r \cap B_t|} \right)}, 
\,\sup\limits_{r\geq t} \frac{1}{\Phi^{-1} \left(\frac{|B_r|^\lambda}{|B_r \cap B_t|} \right)} \Big] \\
&=& \max \Big[ \sup\limits_{r\le t} \frac{1}{\Phi^{-1} \left(|B_r|^{\lambda-1} \right)}, \,
\sup\limits_{r\geq t} \frac{1}{\Phi^{-1} \left(\frac{|B_r|^\lambda}{|B_t|} \right)} \Big] = \dfrac{1}{\Phi^{-1}(|B_t|^{\lambda - 1})},
\end{eqnarray*}
and point (iii) of the lemma has been proved.

(iv) For $t > 0$ we have
\begin{equation*}
\sup_{u>0} \Phi(\dfrac{u}{\varepsilon}) \, \dfrac {1}{|B_r|^{\lambda}} \, |\{x \in B_r \colon \chi_{B_t} (x) > u\}| 
= \sup_{0 < u < 1} \Phi(\dfrac{u}{\varepsilon}) \, \dfrac {|B_r \cap B_t|} {|B_r|^{\lambda}} 
= \Phi(\dfrac{1}{\varepsilon}) \, \dfrac {|B_r \cap B_t|} {|B_r|^{\lambda}}.
\end{equation*}
Thus,
\begin{eqnarray*}
\| \chi_{B_t}\|_{\Phi, \lambda, B_r, \infty} 
&=& 
\inf \{\varepsilon > 0 \colon \Phi(\frac{1}{\varepsilon}) \, \dfrac {|B_r \cap B_t|} {|B_r|^{\lambda}} \leq 1 \} \\
&\leq& 
\inf \{\varepsilon > 0 \colon \frac{1}{\varepsilon} \leq  \Phi^{-1}(\frac{|B_r|^{\lambda}}{|B_r \cap B_t|}) \} 
\leq 1/ {\Phi^{-1} (\frac{|B_r|^\lambda}{|B_r \cap B_t|}) },
\end{eqnarray*}
because $1/\varepsilon \leq  \Phi^{-1}(\Phi (1/\varepsilon))$. On the other hand, since 
$1 \geq \Phi (\frac{1}{\| \chi_{B_t}\|_{\Phi, \lambda, B_r, \infty} }) \frac{|B_r \cap B_t|}{|B_r|^\lambda}$ 
it follows that 
$$ 
\frac{1}{\| \chi_{B_t}\|_{\Phi, \lambda, B_r, \infty} } \leq \Phi^{-1} \left(\Phi( \frac{1}{\| \chi_{B_t}\|_{\Phi, \lambda, B_r, \infty} }) \right) 
\leq \Phi^{-1} \left( \frac{|B_r|^\lambda}{|B_r \cap B_t|} \right), 
$$
which together gives the first equality in (iv). The second equality in (iv) has the same proof as the second equality in (iii).
\end{proof}

\section{On the norm of the dilation operator in central \newline Morrey--Orlicz spaces}

For any $a > 0$ and $x\in \mathbb{R}^n$ we define the {\it dilation operator} $D_a$ by 
$$
D_af(x) = f(ax), \, \, f \in L^0( \mathbb{R}^n).
$$
The dilation operator is bounded in central Morrey--Orlicz spaces $M^{\Phi,\lambda} (0)$ and we will calculate its norm. 
For this purpose quantity $s_{\Phi^{-1}} $ is needed for the Orlicz function $\Phi$:
\begin{equation}\label{def6}
s_{\Phi^{-1}} (t) = \sup\limits_{s>0} \frac{\Phi^{-1}(st)}{\Phi^{-1}(s)} , \qquad~ t > 0.
\end{equation} 

\begin{theorem}\label{dilationOperator}
If $\Phi$ is an Orlicz function, $0 \le \lambda \le 1$ and $a > 0$, then the operator norm of $D_a$ is
\begin{equation}\label{dilOpNorm}
\|D_a\|_{M^{\Phi,\lambda} (0) \rightarrow M^{\Phi,\lambda}(0)}  =  s_{\Phi^{-1}} \left(a^{n(\lambda-1)}\right). 
\end{equation}
\end{theorem}
\begin{proof}
By definition of $s_{\Phi^{-1}}$, for any $s > 0, a > 0$, we have
$$
\Phi^{-1}\left(a^{n(\lambda-1)}s\right) \leq s_{\Phi^{-1}} \left(a^{n(\lambda-1)}\right) \Phi^{-1}(s),
$$
and so
$$
\Phi \left(\frac{\Phi^{-1}\left(a^{n(\lambda-1)} s\right)}{s_{\Phi^{-1}} \left(a^{n(\lambda - 1)}\right)} \right) \leq \Phi \left(\Phi^{-1}(s)\right) = s.
$$
For $a^{n(\lambda - 1)}s = \Phi(u)$ we have $u =\Phi^{-1}\left(a^{n(\lambda-1)} s \right)$ and 
\begin{equation} \label{dil1}
\Phi \left(\frac{u}{s_{\Phi^{-1}} \left(a^{n(\lambda - 1)}\right)} \right) \le a^{n(1-\lambda)} \Phi(u), \quad \text{for any}~ u > 0.
\end{equation}
Therefore, from \eqref{dil1} it follows that for any $f \in {M^{\Phi,\lambda}}(0)$ and $r>0$,
\begin{align*}
\int\limits_{B_r} \Phi \left(\frac{|D_af(x)|}{s_{\Phi^{-1}} \left(a^{n(\lambda - 1)}\right) \|f\|_{M^{\Phi,\lambda}(0)}} \right) \,dx &=
\int\limits_{B_r} \Phi \left(\frac{|f(ax)|}{s_{\Phi^{-1}} \left(a^{n(\lambda - 1)}\right) \|f\|_{M^{\Phi,\lambda}(0)}} \right) \,dx \\
& = a^{-n}\int\limits_{B_{ar}} \Phi \left(\frac{|f(y)|}{s_{\Phi^{-1}} \left(a^{n(\lambda - 1)}\right) \|f\|_{M^{\Phi,\lambda}(0)}} \right) \,dy\\
&  \le a^{-n} a^{n(1-\lambda) } \int\limits_{B_{ar}} \Phi \left(\frac{|f(y)|}{\|f\|_{M^{\Phi,\lambda}(0)}} \right) \,dy 
\le a^{-\lambda n}|B_{ar}|^\lambda\\ & = a^{-\lambda n} v_n^\lambda (ar)^{\lambda n} = |B_r|^{\lambda},
\end{align*}
which means that $\|D_a f\|_{M^{\Phi,\lambda}(0)} \leq s_{\Phi^{-1}} \left(a^{n(\lambda - 1)}\right)  \|f\|_{M^{\Phi,\lambda}(0)}$. 
Here, $v_n = |B_1|$.

To show that \eqref{dilOpNorm} holds we consider the characteristic function $\chi_{B_t}(x)$ of the ball $B_t,~t>0.$ Note that $
D_a \chi_{B_t} (x) = \chi_{B_{t/a}}(x)$. Moreover, by Lemma 1(iii) we get
\begin{align*}
\sup\limits_{t>0} \frac{\|D_a \chi_{B_t}\|_{M^{\Phi,\lambda}(0)}}{\|\chi_{B_t}\|_{M^{\Phi,\lambda}(0)}} &= 
\sup\limits_{t>0} \frac{\Phi^{-1}(|B_t|^{\lambda - 1})}{\Phi^{-1}(|B_{t/a}|^{\lambda - 1})} = 
\sup\limits_{t>0} \frac{\Phi^{-1}(v_n^{\lambda - 1} t^{n(\lambda - 1)})}{\Phi^{-1}\left(v_n^{\lambda - 1} (\frac t a)^{n(\lambda - 1)}\right)}\\
&= \sup\limits_{s>0} \frac{\Phi^{-1}(s)}{\Phi^{-1}(a^{n(1-\lambda)}s)} = 
\sup\limits_{s>0} \frac{\Phi^{-1}(s a^{n(\lambda - 1)})}{\Phi^{-1}(s)} = s_{\Phi^{-1}} (a^{n(\lambda - 1)}).
\end{align*}
This brings us to \eqref{dilOpNorm}.
\end{proof}

\section{The Riesz potential in central Morrey--Orlicz spaces -- necessary conditions} \label{sec:Rp}

We begin to study the boundedness of the Riesz potential, first finding the necessary conditions for its boundedness.

\begin{theorem} \label{necessity}
Let $0 < \alpha < n, \Phi$, $\Psi$ be Orlicz functions and $0 \leq \lambda, \mu < 1$. 
\begin{itemize}
\item[$({\rm i})$] If the Riesz potential $I_{\alpha}$ is bounded from $M^{\Phi, \lambda}(0)$ to $M^{\Psi, \mu}(0)$, 
then there are positive constants $C_1,C_2$ such that
\begin{itemize}
\item[$({\rm a})$] \, \, $u^{\frac{\alpha}{n}} \,\Phi^{-1}(u^{\lambda -1}) \leq C_1\, \Psi^{-1}(u^{\mu-1})\,\,\text{for any}~ u > 0.$
\item[$({\rm b})$] \, $s_{\Psi^{-1}}(u^{\mu-1}) \leq C_2\, u^{\frac{\alpha}{n}} \, s_{\Phi^{-1}}(u^{\lambda-1}) \,\,\text{for any}~ u > 0.$
\end{itemize}
\item[$({\rm ii})$] \, If there exists a small constant $c > 0$ such that $c \leq \frac{v_{n}^{\lambda/\mu } }{v_{n - 1}}$ with $v_0 =1$ 
and
$$
\liminf\limits_{t \rightarrow \infty} \frac{\Phi^{-1}(c t^\lambda)}{\Psi^{-1}(t^\mu)} = \infty,
$$
then $I_{\alpha}$ is not bounded from $M^{\Phi, \lambda}(0)$ to $M^{\Psi, \mu}(0)$.
\end{itemize}
\end{theorem}
\begin{proof}
(i) (a) Let $t > 0$ and $x \in B_t$. In this case we have
$$
I_\alpha \chi_{B_t} (x) = \int\limits_{B_t} |x-y|^{\alpha - n} \,dy \ge (2t)^{\alpha-n} |B_t| = v_n 2^{\alpha-n} t^\alpha
$$
or
$$
t^\alpha \, \chi_{B_t} (x) \leq \frac{2^{n-\alpha}}{v_n} \, I_\alpha \chi_{B_t} (x) \, \chi_{B_t} (x).
$$
Then
$$
\| t^\alpha \, \chi_{B_t} \|_{M^{\Psi, \mu}(0)} \leq \frac{2^{n-\alpha}}{v_n} \,\| I_\alpha \chi_{B_t} \|_{M^{\Psi, \mu}(0)} 
\leq \frac{2^{n-\alpha}}{v_n}\, C \,\| \chi_{B_t} \|_{M^{\Phi, \lambda}(0)},
$$
and by the Lemma \ref{lemma1} (iii) we obtain
$$
\frac{t^\alpha}{\Psi^{-1}(|B_t|^{\mu - 1})} \leq  \frac{2^{n-\alpha}}{v_n}\, C \, \frac{1}{\Phi^{-1}(|B_t|^{\lambda - 1})}, 
$$
which means
$$
\frac{t^\alpha}{\Psi^{-1}(v_n^{\mu - 1} t^{(\mu - 1) n})} \leq  \frac{2^{n-\alpha}}{v_n} \, \frac{C}{\Phi^{-1}(v_n^{\lambda -1} t^{(\lambda - 1) n})}. 
$$
Thus,
$$
t^{\alpha/n} \Phi^{-1}(v_n^{\lambda -1} t^{\lambda - 1}) \leq \frac{2^{n-\alpha}}{v_n} \, C\, \Psi^{-1}(v_n^{\mu - 1} t^{\mu - 1}), 
$$
which by a simple change of variables can be rewritten as 
$$
u^{\alpha/n} \Phi^{-1}(u^{\lambda - 1}) \leq C_1\, \Psi^{-1}(u^{\mu - 1}) \,\, {\rm for ~ any} ~ u > 0,
$$
where $C_1 = 2^{n - \alpha} v_n^{\alpha/n - 1} C$.

(i) (b) First, note that we have identity 
$$
I_\alpha(D_t f)(x) = t^{- \alpha} D_t(I_\alpha f)(x)  \,\,\,  {\rm for ~ any} ~ t > 0.
$$
In fact, 
$$
I_\alpha(D_t f)(x) = \int\limits_{\mathbb{R}^n} \frac{f(ty)}{|x-y|^{n-\alpha}}\,dy = 
t^{-\alpha} \, \int\limits_{\mathbb{R}^n} \frac{f(y)}{|y - t x|^{n-\alpha}} \, dy = t^{-\alpha} D_t(I_\alpha f) (x).
$$
Now, let $f \in M^{\Phi, \lambda}(0)$. Using the above identity and applying Theorem \ref{dilationOperator} we obtain
$$
\| I_\alpha(D_t f)\|_{M^{\Psi, \mu}(0)} = t^{-\alpha} \, \| D_t(I_\alpha f)\|_{M^{\Psi, \mu}(0)} 
=  t^{-\alpha} \, s_{\Psi^{-1}}(t^{n(\mu - 1)})  \, \| I_\alpha f\|_{M^{\Psi, \mu}(0)}.
$$
Assumption of boundedness of $ I_\alpha$ and reuse of Theorem \ref{dilationOperator} gives
\begin{eqnarray*}
\| I_\alpha f\|_{M^{\Psi, \mu}(0)} 
&=& 
\frac{t^{\alpha}}{ s_{\Psi^{-1}}(t^{n(\mu - 1)}) } \| I_\alpha(D_t f)\|_{M^{\Psi, \mu}(0)} \\
&\leq&
\frac{t^{\alpha}}{s_{\Psi^{-1}}(t^{n(\mu - 1)})} \, C\, \| D_t f \|_{M^{\Phi, \lambda}(0)} \\
&=&
C\, \frac{t^{\alpha}}{s_{\Psi^{-1}}(t^{n(\mu - 1)})} s_{\Phi^{-1}}(t^{n(\lambda - 1)}) \, \| f\|_{M^{\Phi, \lambda}(0)},
\end{eqnarray*}
or
$$
\| I_\alpha f\|_{M^{\Psi, \mu}(0)} \leq C \,\frac{u^{\alpha/n} s_{\Phi^{-1}}(u^{\lambda - 1}) }{s_{\Psi^{-1}}(u^{\mu - 1}) }  
\, \| f\|_{M^{\Phi, \lambda}(0)}  \,\,\,  {\rm for ~ any} ~ u> 0.
$$
Thus, 
$$
\| I_\alpha f\|_{M^{\Psi, \mu}(0)} \leq C \, \inf_{u> 0} \frac{u^{\alpha/n} s_{\Phi^{-1}}(u^{\lambda - 1}) }{s_{\Psi^{-1}}(u^{\mu - 1}) } 
\, \| f\|_{M^{\Phi, \lambda}(0)} .
$$
We must have that $ \inf_{u> 0} \frac{u^{\alpha/n} s_{\Phi^{-1}}(u^{\lambda - 1}) }{s_{\Psi^{-1}}(u^{\mu - 1}) } = c > 0$ since 
otherwise $I_\alpha f = 0$ and we get a contradiction. Therefore,
$$
s_{\Psi^{-1}}(u^{\mu - 1}) \leq \frac{C}{c} \, u^{\alpha/n} s_{\Phi^{-1}}(u^{\lambda - 1})  \,\,\,  {\rm for ~ any} ~ u> 0.
$$
(ii) We follow the same argument as in \cite[Proposition 1]{KS17}. Let $R \ge 1,~x_R = (R,0,...,0) \in \mathbb{R}^n$ 
and $f_R(x) = \chi_{B(x_R,1)}(x).$ Then 
\begin{eqnarray*}
\|f_R\|_{M^{\Phi, \lambda}(0)} &=& \sup\limits_{r>0} \inf \Bigl\{\varepsilon>0 \colon \frac{1}{|B_r|^\lambda}\int\limits_{B_r} 
\Phi\left(\frac{\chi_{B(x_R,1)(x)}}{\varepsilon}\right)\,dx \le 1\Bigr\}\\
&=& \sup\limits_{r>0} \inf \Bigl\{\varepsilon>0 \colon \frac{1}{|B_r|^\lambda}\int\limits_{B_r \cap B(x_R,1)} 
\Phi\left(\frac{1}{\varepsilon}\right)\,dx \le 1\Bigr\}\\
&=& \sup\limits_{r>0} \inf \Bigl\{\varepsilon>0 \colon \frac{|B_r \cap B(x_R,1)|}{|B_r|^\lambda}\Phi\left(\frac{1}{\varepsilon}\right)
 \le 1\Bigr\}\\
&=& \sup\limits_{r>R-1} \inf \Bigl\{\varepsilon>0 \colon \frac{|B_r \cap B(x_R,1)|}{|B_r|^\lambda}\Phi\left(\frac{1}{\varepsilon}\right) 
\le 1\Bigr\},
\end{eqnarray*}
because if $0<r \le R-1$ then $|B_r \cap B(x_R,1)| = 0$. Thus,
$$
\|f_R\|_{M^{\Phi, \lambda}(0)} = \sup\limits_{r>R-1} \frac{1}{\Phi^{-1}(\frac{|B_r|^\lambda}{|B_r \cap B(x_R,1)|})}.
$$
We will consider two cases: $R-1 < r < R$ and $r \ge R.$ In the first case, using calculations from \cite[p. 161]{BG04}, we 
can prove that for $n \geq 2$
$$
|B_r \cap B(x_R,1)| \le 2^{\frac n 2} v_{n-1} \left(\frac{r}{R}\right)^n,
$$ 
and so
$$
\frac{|B_r|^\lambda}{|B_r \cap B(x_R,1)|} \ge \frac{v_n^\lambda r^{\lambda n} R^n}{2^{\frac n 2} v_{n-1} r^n} \ge 
\frac{v_n^\lambda}{2^{\frac n 2} v_{n-1}}R^{\lambda n} > \frac{v_n^\lambda}{2^n \,v_{n-1}}R^{\lambda n}.
$$
For $n = 1$ and $R - 1 < r < R$ with $v_0 = 1$ we have
\begin{eqnarray*}
\frac{|B_r|^\lambda}{|B_r \cap B(x_R,1)|} &= \frac{(2 r)^{\lambda}}{r - R + 1} = \frac{2^{\lambda} r^{\lambda - 1}}{1 - \frac{R - 1}{r}} 
> \frac{2^{\lambda} R^{\lambda - 1}}{1 - \frac{R - 1}{r}} \\
&= \frac{2^{\lambda} R^{\lambda}}{R - \frac{R (R - 1)}{r}} > 2^{\lambda} R^{\lambda} = v_1^{\lambda} R^{\lambda} 
>  \frac{v_1^{\lambda}}{2 \,v_0} R^{\lambda}.
\end{eqnarray*}
In the second case, $|B_r \cap B(x_R,1)| \le |B(x_R,1)| = v_n$ and
$$
\frac{|B_r|^\lambda}{|B_r \cap B(x_R,1)|} \ge \frac{v_n^\lambda r^{\lambda n}}{v_n} \ge v_n^{\lambda - 1} R^{\lambda n}.
$$
Thus,
$$
\|f_R\|_{M^{\Phi, \lambda}(0)} \le \max \Biggl[\frac{1}{\Phi^{-1}\left(\frac{v_n^\lambda}{2^n \,v_{n-1}}R^{\lambda n}\right)}, 
\frac{1}{\Phi^{-1}(v_n^{\lambda - 1}R^{\lambda n})} \Biggr].
$$
Since $\frac {v_{n-1}}{v_n} \ge \sqrt{\frac{n}{2 \pi}}$ with $v_0 = 1$ (see \cite[Theorem 2]{Al00}), it follows that 
$\frac {2^n \,v_{n-1}}{v_n} \geq 1$ and then $\frac{v_n^\lambda}{2^n \, v_{n-1}} \leq v_n^{\lambda - 1}$, which gives
$$
\|f_R\|_{M^{\Phi, \lambda}(0)} \le \frac{1}{\Phi^{-1}\left(\frac{v_n^\lambda}{2^n\, v_{n-1}}R^{\lambda n}\right)}.
$$
Next, we will estimate $I_\alpha f_R.$ If $x,y \in B(x_R,1)$ then $|x-y| \le 2$ and we obtain
\begin{eqnarray*}
I_\alpha f_R (x)&=&\int\limits_{\mathbb{R}^n} \frac{\chi_{B(x_R,1)}(y)}{|x-y|^{n-\alpha}}\,dy = 
\int\limits_{B(x_R,1)} |x-y|^{\alpha-n} \,dy \\
& \ge& 2^{\alpha-n} |B(x_R,1)| \, \chi_{B(x_R,1)}(x) = 2^{\alpha-n} v_n\, \chi_{B(x_R,1)}(x).
\end{eqnarray*}
Thus,
\begin{eqnarray*}
\|I_\alpha f_R\|_{M^{\Psi, \mu}(0)} & = & \sup\limits_{r>0} \|I_\alpha f_R\|_{\Psi, \mu, B_r} \ge \|I_\alpha f_R\|_{\Psi, \mu, B_{R+1}}\\
& = & \inf\biggl\{ \varepsilon>0 \colon \int\limits_{B_{R+1}} \Psi \left(\frac{|I_\alpha f_R (x)|}{\varepsilon} \right) \,dx \le |B_{R+1}|^\mu \biggr\}.
\end{eqnarray*}
Since $x \in B(x_R,1)$ and $B_{R+1} \cap B(x_R,1) = B(x_R,1)$ it follows that
\begin{eqnarray*}
\|I_\alpha f_R\|_{M^{\Psi, \mu}(0)} & \ge & \inf\biggl\{ \varepsilon>0 \colon \int\limits_{B(x_R,1)  \cap B_{R+1}} 
\Psi \biggl(2^{\alpha -n} v_n / \varepsilon \biggr) \,dx \le |B_{R+1}|^\mu \biggr\}\\
& = & 
\frac{2^{\alpha -n} v_n}{\Psi^{-1}\left(\frac{|B_{R+1}|^\mu}{|B(x_R,1)|} \right)} = \frac{2^{\alpha -n} v_n}{\Psi^{-1} \left(v_n^{\mu - 1} (R +1)^{\mu n} \right)}
\geq \frac{2^{\alpha -n} v_n} {\Psi^{-1} \left(v_n^{\mu - 1} 2^{\mu n} R^{\mu n} \right)}.
\end{eqnarray*}
Making the substitution $t^\mu = v_n^{\mu-1}2^{\mu n}R^{\mu n}$ we obtain
\begin{eqnarray*}
\frac{\|I_\alpha f_R\|_{M^{\Psi, \mu}(0)}}{\|f_R\|_{M^{\Phi, \lambda}(0)}} &\ge& 2^{\alpha -n} \,v_n 
\frac{\Phi^{-1}\left(\frac{v_n^\lambda}{2^n v_{n-1}} R^{\lambda n}\right)}{\Psi^{-1}(v_n^{\mu-1} 2^{\mu n}R^{\mu n})} = 
2^{\alpha -n} \, v_n \frac{\Phi^{-1}\left(\frac{v_n^{\frac {\lambda}{\mu}}}{2^{n +\lambda n} \,v_{n-1} } \, t^\lambda\right)}{\Psi^{-1}(t^\mu)}\\
&\ge& \frac{2^{\alpha - n} \, v_n}{2^{n +\lambda n}} \frac{\Phi^{-1}\left(\frac{v_n^{\frac {\lambda}{\mu}}}{v_{n-1} } 
\, t^\lambda\right)}{\Psi^{-1}(t^\mu)} \geq 2^{\alpha - 2 n - \lambda n} \, v_n \frac{\Phi^{-1}\left(c \, t^\lambda\right)}{\Psi^{-1}(t^\mu)},
\end{eqnarray*}
and
$$
\liminf\limits_{R \rightarrow \infty} \frac{\|I_\alpha f_R\|_{M^{\Psi, \mu}(0)}}{\|f_R\|_{M^{\Phi, \lambda}(0)}} \geq 
2^{\alpha - 2 n - \lambda n} \, v_n\, \liminf\limits_{t \rightarrow \infty} \frac{\Phi^{-1}(c t^\lambda)}{\Psi^{-1}(t^\mu)} = \infty.
$$
Thus, the operator $I_\alpha$ is not bounded from $M^{\Phi, \lambda}(0)$ to $M^{\Psi, \mu}(0)$ for $n \geq 1$.
\end{proof}

\section{The Riesz potential in central Morrey--Orlicz spaces -- sufficient conditions} \label{sec:6}

We want to prove boundedness of the Riesz potential $I_{\alpha}$ between two different central 
Morrey--Orlicz spaces. The following lemmas are important for proving the main result. 

\begin{lemma}\label{lemma3} 
Let $0< \alpha <n, \Phi$ be an Orlicz function and $0 \leq \lambda < 1$. If $f \in M^{\Phi,\lambda}(0)$, then there 
exists a constant $C_3 > 0$ such that 
\begin{equation*}
\int_{{\mathbb R^n}\setminus B_{r}} \dfrac {|f(y)|} {|y|^{n-\alpha}} \,dy \le C_3 \, \|f\|_{M^{\Phi,\lambda}(0)}  
\int_{|B_r|}^{\infty} t^{\alpha/n} \Phi^{-1} ( t^{\lambda-1}) \, \frac{dt}{t} 
\end{equation*}
for all $r > 0$.
\end{lemma}

\begin{proof} We prove this lemma using the same arguments as in the proof of Theorem 7.1 in \cite{Na08a} 
and Lemma 2.5 in \cite{MM20}. From the Lemma \ref{lemma1} (i) and (ii) it follows that
\begin{eqnarray*}
\int_{{\mathbb R^n}\setminus B_{r}} \dfrac {|f(y)|} {|y|^{n-\alpha}} \,dy 
& =& 
\sum_{j=1}^{\infty} \int_{B_{2^jr} \setminus B_{2^{j-1}r}} \dfrac {|f(y)|} {|y|^{n-\alpha}} \,dy  
\leq \sum_{j=1}^{\infty} \frac{1}{(2^{j-1}r)^{n-\alpha}} \int_{B_{2^jr} } |f(y)| \,dy\\
& =& 
2^{n-\alpha} v_n^{1-\alpha/n} \sum_{j=1}^{\infty} \dfrac 1 {|B_{2^jr}|^{1-\alpha/n}} \int_{B_{2^jr}} |f(y)| \,dy \\
& \leq& 
2^{n-\alpha+1} v_n^{1-\alpha/n} \sum_{j=1}^{\infty} |B_{2^jr}|^{\lambda-1+\alpha/n} \|f\|_{\Phi,\lambda, B_{2^jr}} \|1\|_{\Phi^*,\lambda,B_{2^jr}} \\ 
& \leq& 
C_3^{\prime} \sum_{j=1}^{\infty} |B_{2^jr}|^{\lambda-1+\alpha/n} \|f\|_{\Phi,\lambda, B_{2^jr}} 
                  \dfrac {\Phi^{-1}(|B_{2^{j} r}|^{\lambda-1})} {|B_{2^jr}|^{\lambda-1}} \\ 
& =& 
\frac{C_3^{\prime}}{n \ln 2} \, \|f\|_{M^{\Phi,\lambda}(0)} \,\sum_{j=1}^{\infty} |B_{2^jr}|^{\alpha/n} \Phi^{-1}(|B_{2^jr}|^{\lambda-1}) 
\int_{ |B_{2^{j - 1} r}|}^{ |B_{2^jr}|} \frac{dt}{t}\\ 
& \leq& 
\frac{C_3^{\prime}}{n \ln 2} \, 2^\alpha  \, \|f\|_{M^{\Phi,\lambda}(0)} \, \sum_{j=1}^{\infty}
\int_{ |B_{2^{j - 1} r}|}^{ |B_{2^jr}|} t^{\alpha/n} \Phi^{-1} ( t^{\lambda-1}) \, \frac{dt}{t}\\
& \leq& 
C_3 \, \|f\|_{M^{\Phi,\lambda}(0)}  \int_{|B_r|}^{\infty} t^{\alpha/n} \Phi^{-1} ( t^{\lambda-1}) \, \frac{dt}{t},
\end{eqnarray*}
where $C_3^{\prime} = 2^{n-\alpha+1} v_n^{1-\alpha/n}$ and $C_3 = \frac{2^{\alpha}}{n \ln 2} C_3^{\prime} $. 
Thus, we arrive to the assertion of Lemma \ref{lemma3}.
\end{proof}

Next, we show the following well-definedness of $I_{\alpha}f$ when $f \in M^{\Phi, \lambda}(0)$. 

\begin{lemma} \label{lemma2} Let $0 < \alpha < n$, $\Phi$ be an Orlicz function and $0 \leq \lambda <1$. 
If the integral $\int_{|B_r|}^{\infty} t^{\alpha/n} \Phi^{-1}(t^{\lambda - 1})\frac {dt}{t}$ is convergent for any $r>0$ and $f \in M^{\Phi, \lambda}(0)$ then the Riesz potential $I_{\alpha}f$ is well-defined. 
\end{lemma}

\begin{proof} We will prove this lemma using the same arguments that were presented in the proof 
in \cite[Theorem 2.1]{MN11}. 
Let $f \in M^{\Phi,\lambda}(0)$, $r > 0$ and $x \in B_r$, 
and let 
\begin{equation} \label{defined}
I_{\alpha}f(x) = I_{\alpha}(f\chi_{B_{2r}})(x) + I_{\alpha}(f(1-\chi_{B_{2r}}))(x). 
\end{equation}
Since $f\chi_{B_{2r}} \in L^1({\mathbb R^n})$, the first term is well-defined. Indeed, in view of \cite[Theorem 1.1, Chapter 2]{Miz} the requirement $I_\alpha |f\chi_{B_{2r}}| \not\equiv \infty$ for any $f \in M^{\Phi,\lambda}(0)$ and $r>0$ is equivalent to
$$
\int_{B_{2r}}(1+|y|)^{\alpha-n}|f(y)|\,dy < \infty.
$$
The last inequality is true since $\|(1+|y|)^{\alpha-n}\|_{\Phi^\ast, \lambda, B_{2r}} \le \frac{(1+2r)^{\alpha-n}}{(\Phi^\ast)^{-1}(|B_{2r}|^{\lambda-1})}$ and by Lemma \ref{lemma1} (i) we obtain
\begin{align*}
\int_{B_{2r}}(1+|y|)^{\alpha-n}|f(y)|\,dy &\le 2 |B_{2r}|^\lambda \|f\|_{\Phi, \lambda, B_{2r}} \|(1+|y|)^{\alpha-n}\|_{\Phi^\ast, \lambda, B_{2r}}\\
&\le2|B_{2r}|^\lambda \frac{(1+2r)^{\alpha-n}}{(\Phi^\ast)^{-1}(|B_{2r}|^{\lambda-1})}\|f\|_{M^{\Phi, \lambda}(0)} < \infty.
\end{align*}
For the second term for any $x \in B_r$ we have
\begin{eqnarray*}
|I_{\alpha}(f(1-\chi_{B_{2r}}))(x)| &\le& \int_{{\mathbb R^n}\setminus B_{2r}} \frac{|f(y)|}{|x-y|^{n-\alpha}} \,dy \le 2^{n-\alpha} \int_{{\mathbb R^n}\setminus B_{2r}} \frac{|f(y)|}{|y|^{n-\alpha}} \,dy.
\end{eqnarray*}
Since the integral $\int_{|B_r|}^{\infty} t^{\alpha/n} \Phi^{-1}(t^{\lambda - 1})\frac {dt}{t}$ is convergent for any $r>0$ and $f \in M^{\Phi, \lambda}(0)$ it follows from Lemma \ref{lemma3} that $I_{\alpha}(f(1-\chi_{B_{2r}}))(x)$ is well-defined for all $x \in B_r.$

Further, since for $0 < s < r$, 
$$
f\chi_{B_{2s}}+f(1-\chi_{B_{2s}})=f\chi_{B_{2r}}+f(1-\chi_{B_{ 2 r}}), 
$$
it follows that for $x \in B_s \subset B_r$, 
$$
I_{\alpha}(f\chi_{B_{2s}})(x)+I_{\alpha}(f(1-\chi_{B_{2s}}))(x) = I_{\alpha}(f\chi_{B_{2r}})(x)+I_{\alpha}(f(1-\chi_{B_{ 2 r}}))(x). 
$$
This shows that $I_{\alpha}f$ is independent of $B_r$ containing $x$. 
Thus, $I_{\alpha}f$ is well-defined on ${\mathbb R^n}$. 
\end{proof}

Now we will present sufficient conditions on spaces so that the operator $I_\alpha$ is bounded between 
distinct central Morrey--Orlicz spaces. In the proofs of these estimates we will use estimates from 
\cite{MM15} for the Hardy--Littlewood maximal operator. The {\it Hardy--Littlewood maximal operator} $M$ 
is defined for $f \in L^1_{loc} ({\mathbb R^n})$ and $x {\in \mathbb R^n}$ by
$$
Mf(x) = \sup\limits_{r>0} \frac {1} {|B(x,r)|} \int\limits_{B(x,r)} |f(y)| \,dy. 
$$
Then, for an Orlicz function $\Phi$ and $0 \leq \lambda \leq 1$, this operator $M$ is bounded on $M^{\Phi, \lambda}(0)$, 
provided ${\Phi}^* \in {\Delta}_2$, that is, there exists a constant $C_0>1$ such that 
\begin{equation} \label{max}
\|Mf\|_{M^{\Phi, \lambda}(0)} \le C_0 \,  \|f\|_{M^{\Phi, \lambda}(0)}  \quad \text{for all} ~ f \in M^{\Phi, \lambda}(0)
\end{equation} 
(see \cite[Theorem 6(i)]{MM15}). Moreover, $M$ is bounded from  $M^{\Phi, \lambda}(0)$ to $WM^{\Phi, \lambda}(0)$, 
that is, there exists a constant $c_0>1$ such that $\|Mf\|_{WM^{\Phi, \lambda}(0)} \le c_0 \,  
\|f\|_{M^{\Phi, \lambda}(0)}$ for all $f \in M^{\Phi, \lambda}(0)$ (see \cite[Theorem 6(ii)]{MM15}).

\begin{theorem} \label{Thm3}
Let $0< \alpha <n, \Phi, \Psi$ be Orlicz functions and either $0 < \lambda, \mu < 1, \lambda \ne \mu$ or $\lambda = \mu = 0$. 
Assume that there exist constants $C_4, C_5 \ge 1$ such that 
\begin{equation} \label{Phi14}
\int_u^{\infty} t^{\frac{\alpha}{n}} \, \Phi^{-1}(t^{\lambda-1}) \, \frac{dt}{t} \leq C_4\, \Psi^{-1}(u^{\mu-1}) \quad \text{for all} \ u > 0 
\end{equation}
and
\begin{equation} \label{Orlicz15} 
\int_u^{\infty} t^\frac{\alpha}{n} \, \Phi^{-1} (\frac{r^\lambda}{t})\, \frac{dt}{t} \leq C_5 \,\Psi^{-1}(\frac{r^\mu}{u}) 
\quad \text{for all} \  u > 0~\text{and for all}~r>0. 
\end{equation}
\begin{itemize}
\item[$({\rm i})$] If ${\Phi}^* \in {\Delta}_2$, then $I_{\alpha}$ is bounded from $M^{\Phi, \lambda}(0)$ to $M^{\Psi, \mu}(0)$, 
that is, there exists a constant $C_6 \geq 1$ such that $\| I_{\alpha} f\|_{M^{\Psi, \mu}(0)} \leq C_6 \,  \|f\|_{M^{\Phi, \lambda}(0)}$ 
for all $f \in M^{\Phi, \lambda}(0)$. 
\item[$({\rm ii})$] The operator $I_{\alpha}$ is bounded from $M^{\Phi, \lambda}(0)$ to $WM^{\Psi, \mu}(0)$, that is, 
there exists a constant $c_6 \geq 1$ such that $\| I_{\alpha} f\|_{WM^{\Psi, \mu}(0)} \leq c_6\,  \|f\|_{M^{\Phi, \lambda}(0)}$ 
for all $f \in M^{\Phi, \lambda}(0)$. 
\end{itemize}
\end{theorem}

\begin{remark} 
The same conclusions hold for non-homogeneous versions of $M^{\Phi, \lambda}(0)$ and $M^{\Psi, \mu}(0)$. 
\end{remark}

\begin{remark} 
From the estimate (\ref{Phi14}) we get the inequality (a) in Theorem \ref{necessity}(i).
Namely, using the concavity of the function $\Phi^{-1}$ we get
$$
\int\limits_u^\infty t^{\frac{\alpha}{n}} \Phi^{-1} (t^{\lambda - 1}) \frac{dt}{t} \ge 
\int\limits_u^{2u} t^{\frac{\alpha}{n}} \Phi^{-1} (t^{\lambda - 1}) \frac{dt}{t} 
\ge u^{\frac{\alpha}{n}} \Phi^{-1}((2 u)^{\lambda - 1}) \ln{2} \ge 
\frac{\ln 2}{2^{1 - \lambda}}~u^{\frac{\alpha}{n}} \Phi^{-1} (u^{\lambda - 1}).
$$
\end{remark}

\begin{remark} 
Note that if either $\lambda = \mu >  0$ or $\lambda = 0$ and $\mu > 0$, then estimate (\ref{Orlicz15}) doesn't hold.
\end{remark}

\begin{remark} 
If $\lambda = \mu = 0$, then inequalities (\ref{Phi14}) and (\ref{Orlicz15}) are the same. Moreover, condition 
(\ref{Phi14}) in this case is a sufficient condition for boundedness of $I_{\alpha}$ from Orlicz space 
$L^{\Phi}({\mathbb R^n})$ to weak Orlicz space $WL^{\Psi}({\mathbb R^n})$, and if additionally $\Phi^* \in \Delta_2$
then $I_{\alpha}$ is bounded from Orlicz space $L^{\Phi}({\mathbb R^n})$ to Orlicz space $L^{\Psi}({\mathbb R^n})$ 
(proof we can find, for example, in \cite[Theorem 3.3]{GDH17}). 
\end{remark}

In the proof of Theorem \ref{Thm3} the following lemma plays a crucial role.

\begin{lemma} \label{lemma_case1}
Let $0<\alpha <n, \Phi, \Psi$ be Orlicz functions, $\Phi^* \in \Delta_2$ and either $0 < \lambda, \mu <1, \lambda \neq \mu$ 
or $\lambda = \mu = 0$. If the estimate {\rm(\ref{Orlicz15})} holds, then there exists a constant $C_7 \geq 1$ such that 
\begin{equation*}
\int\limits_{B_r}\Psi\left( \dfrac {\int_{B_{2r}} \dfrac {|f(y)|} {|x-y|^{n-\alpha}} \,dy} {C_7 \, \|f\|_{M^{\Phi,\lambda}(0)}} \right) \,dx 
\le  |B_r|^\mu, \quad~ \text{for all}~ f \in M^{\Phi,\lambda}(0)~ \text{and}~ r>0. 
\end{equation*}
\end{lemma}

\begin{proof}
Let $f \in M^{\Phi,\lambda}(0)$. We write $I_\alpha (f \chi_{B_{2r}})$ as follows
\begin{align*}
I_\alpha (f \chi_{B_{2r}})(x) &= \int_{B_{2r}} \frac {|f(y)|} {|x-y|^{n-\alpha}} \,dy 
= \int_{|x-y| \le \delta} \frac {|f(y)\chi_{B_{2r}}(y)|} {|x-y|^{n-\alpha}} \,dy \\
& + \int_{|x-y| > \delta} \frac {|f(y)\chi_{B_{2r}}(y)|} {|x-y|^{n-\alpha}} \,dy =: J_1f(x)+J_2f(x),
\end{align*}
where $\delta>0$ will be defined later on.
It is known that 
$$
J_1f(x) \le C_8 \, |B_{\delta}|^\frac{\alpha}{n} M(f\chi_{B_{2r}})(x),
$$
where $C_8 = \frac{2^\alpha}{2^\alpha - 1} \, C_3^{\prime}$.
Note that for any parameters $u > 0$ and $r>0$ we have
\begin{align*}
\int\limits_u^\infty t^{\frac{\alpha}{n}} \Phi^{-1} (\frac{r^\lambda}{t}) \frac{dt}{t} &\ge 
\int\limits_u^{2u} t^{\frac{\alpha}{n}} \Phi^{-1} (\frac{r^\lambda}{t}) \frac{dt}{t}\\
&\ge \ln{2}~ u^{\frac{\alpha}{n}} \Phi^{-1}(\frac{r^\lambda}{2u}) \ge 
\frac{\ln 2}{2}~u^{\frac{\alpha}{n}} \Phi^{-1} (\frac{r^\lambda}{u}).
\end{align*}
Thus, applying \eqref{Orlicz15} we obtain
$$
J_1f(x) \le \frac{2}{\ln 2} \, C_5 \, C_8 \, \frac{\Psi^{-1}\left(\frac{|B_{2r}|^\mu}{|B_{\delta}|}\right)}{\Phi^{-1}
\left(\frac{|B_{2r}|^\lambda}{|B_{\delta}|}\right)} M(f\chi_{B_{2r}})(x).
$$
Following Hedberg's method we get for $J_2f(x)$
\begin{align*}
J_2f(x)  &= \sum\limits_{k = 1}^\infty \, \int\limits_{2^{k-1} \delta <|x-y| \le 2^k \delta} \frac {|f(y)\chi_{B_{2r}}(y)|} {|x-y|^{n-\alpha}} \,dy\\
& \le \sum\limits_{k = 1}^\infty (2^{k-1} \delta)^{\alpha - n} \int\limits_{|x-y| \le 2^k \delta} |f(y)\chi_{B_{2r}}(y)| \,dy\\
&=\sum\limits_{k = 1}^\infty (2^{k-1} \delta)^{\alpha - n} \int\limits_{B_{2r}} |f(y)\chi_{B(x,2^k \delta)}(y)| \,dy.
\end{align*}
From Lemma \ref{lemma1} (i) and (ii) it follows that
\begin{align*}
J_2f(x)  &\le 2 \, |B_{2r}|^\lambda \|f\|_{\Phi, \lambda, B_{2r}} \sum\limits_{k = 1}^\infty (2^{k-1} \delta)^{\alpha - n} 
\|\chi_{B(x,2^k\delta)}\|_{\Phi^\ast, \lambda, B_{2r}}\\
&\le 2^{n-\alpha+1} \|f\|_{\Phi, \lambda, B_{2r}} \sum\limits_{k = 1}^\infty (2^k \delta)^{\alpha - n} |B_{2r} \cap 
B(x,2^k \delta)| \Phi^{-1} \left(\frac{|B_{2r}|^\lambda}{|B_{2r} \cap B(x,2^k \delta)|} \right).
\end{align*}
Taking into account that $u\Phi^{-1}(1/u)$ is increasing and $|B_{2r} \cap B(x,2^k \delta)| \le |B(x,2^k \delta)|$ we obtain
\begin{align*}
J_2 f(x) &\le 2^{n-\alpha+1} \|f\|_{\Phi, \lambda, B_{2r}} \sum\limits_{k = 1}^\infty (2^k \delta)^{\alpha - n} \Phi^{-1} 
\left(\frac{|B_{2r}|^\lambda}{|B(x,2^k\delta)|} \right) |B(x,2^k \delta)|\\
& = 2^{n-\alpha+1} v_n \, \|f\|_{\Phi, \lambda, B_{2r}} \sum\limits_{k = 1}^\infty (2^k \delta)^{\alpha} \Phi^{-1} 
\left(\frac{|B_{2r}|^\lambda}{|B_{2^k \delta}|} \right)\\
& =\frac{C_3^{\prime}}{n \ln{2}} \, \|f\|_{\Phi, \lambda, B_{2r}} \sum\limits_{k = 1}^\infty |B_{2^k\delta}|^{\frac{\alpha}{n}} \Phi^{-1}
 \left(\frac{|B_{2r}|^\lambda}{|B_{2^k\delta}|}\right) \int\limits_{|B_{2^{k-1}\delta}|}^{|B_{2^k \delta}|}\frac{dt}{t}\\
&\le \frac{C_3^{\prime}}{n \ln{2}} \, \|f\|_{\Phi, \lambda, B_{2r}} \sum\limits_{k = 1}^\infty |B_{2^{k}\delta}|^{\frac{\alpha}{n}} 
\int\limits_{|B_{2^{k-1}\delta}|}^{|B_{2^k \delta}|}\Phi^{-1}\left(\frac{|B_{2r}|^\lambda}{t}\right) \frac{dt}{t}\\
&\le C_3 \, \|f\|_{\Phi, \lambda, B_{2r}} \int\limits_{|B_{\delta}|}^\infty t^{\frac{\alpha}{n}} \Phi^{-1} 
\left(\frac{|B_{2r}|^\lambda}{t}\right) \frac{dt}{t}\\
&\le C_5 \, C_3 \, \|f\|_{M^{\Phi, \lambda}(0)} \Psi^{-1}\left(\frac{|B_{2r}|^\mu}{|B_{\delta}|}\right).
\end{align*}
Now we choose $\delta>0$ such that
$$
\frac{Mf(x)}{C_0 \,\|f\|_{M^{\Phi, \lambda}(0)}} = \Phi^{-1}\left(\frac{|B_{2r}|^\lambda}{|B_\delta|}\right),
$$
where the constant $C_0$ is from (\ref{max}). Then
$$
J_1 f(x) \leq \frac{2}{\ln 2} \, C_5 \, C_8 \, C_0\, \| f \|_{M^{\Phi, \lambda}(0)} \,\Psi^{-1}\left(\frac{|B_{2r}|^\mu}{|B_{\delta}|}\right),
$$
and 
\begin{align*}
& \int_{B_{2r}} \dfrac {|f(y)|} {|x-y|^{n-\alpha}} \,dy = J_1 f(x) + J_2 f(x) \\
&\le  \left( \frac{2}{\ln 2} \, C_5 \, C_8 \, C_0 + C_5 C_3 \right) \,  \| f \|_{M^{\Phi, \lambda}(0)} 
\,\Psi^{-1}\left(\frac{|B_{2r}|^\mu}{|B_{\delta}|}\right)
\end{align*}
Thus, with $C_{9} = 2\, C_5 \max\left(\frac{2}{\ln 2} C_0 C_8,\, C_3\right)$ we obtain
\begin{align*}
 \int_{B_{2r}} \dfrac {|f(y)|} {|x-y|^{n-\alpha}} \,dy &\le C_{9} \|f\|_{M^{\Phi, \lambda}(0)} \Psi^{-1} \left(|B_{2r}|^{\mu-\lambda} 
\Phi \left( \frac{Mf(x)}{C_0\|f\|_{M^{\Phi, \lambda}(0)}}\right)\right).
\end{align*}
Then 
\begin{align*}
\Psi \left(\frac{\int_{B_{2r}} \dfrac {|f(y)|} {|x-y|^{n-\alpha}} \,dy}{C_{9} \|f\|_{M^{\Phi, \lambda}(0)}} \right)& \le 
|B_{2r}|^{\mu-\lambda} \,\Phi\left(\frac{Mf(x)}{\|Mf\|_{M^{\Phi, \lambda}(0)}}\right)\\
 &= 2^{n(\mu-\lambda)} \, |B_r|^{\mu - \lambda} \, \Phi\left(\frac{Mf(x)}{\|Mf\|_{M^{\Phi, \lambda}(0)}}\right).
\end{align*}
Finally, with $C_7 = 2^{n(\mu-\lambda)} C_{9}$ we get
\begin{align*}
\frac{1}{|B_r|^\mu}\int\limits_{B_r}\Psi \left(\frac{\int_{B_{2r}} \dfrac {|f(y)|} {|x-y|^{n-\alpha}} \,dy}{C_7 \, 
\|f\|_{M^{\Phi, \lambda}(0)}} \right) \,dx & 
\le \frac{1}{|B_r|^\lambda} \int\limits_{B_r} \Phi\left(\frac{Mf(x)}{\|Mf\|_{M^{\Phi, \lambda}(0)}}\right) \,dx \le 1
\end{align*}
and we arrive to the statement of this lemma.
\end{proof}

\begin{proof}[Proof of Theorem \ref{Thm3}]  
(i) Let $0< \alpha <n$ and $0 \le \lambda < 1, 0 < \mu <1$. Let also $f \in M^{\Phi,\lambda}(0)$ and $r >0.$ 
Since $I_{\alpha}f$ is well-defined by Lemma \ref{lemma2}, we prove only that 
$$
\|I_{\alpha}f\|_{M^{\Psi,\mu}(0)} \leq C_6 \, \|f\|_{M^{\Phi,\lambda}(0)}. 
$$
Now, by (\ref{defined}), for $C_6 =2 \,\max(C_7,\, 2^{n-\alpha}C_3\,C_4)$, it follows that 
\begin{eqnarray*}
&& \int_{B_r} \Psi \left( \dfrac {|I_{\alpha}f(x)|}{C_6 \, \|f\|_{M^{\Phi,\lambda}(0)}} \right)\,dx  \\
&& \quad \ \ \ \leq \ 
\dfrac{1}{2} \int_{B_r} \Psi \left( \dfrac {|I_{\alpha}(f\chi_{B_{2r}})(x)|} {C_7 \|f\|_{M^{\Phi,\lambda}(0)}} \right)\,dx 
  + \dfrac{1}{2} \int_{B_r} \Psi \left( \dfrac {|I_{\alpha}(f(1-\chi_{B_{2r}}))(x)|} {2^{n-\alpha} \,C_3 \,C_4\, 
 \|f\|_{M^{\Phi,\lambda}(0)}} \right)\,dx \\
&& \quad \ \ \ =: \ \dfrac{1}{2} \left( I_1 + I_2 \right). 
\end{eqnarray*}
From Lemma \ref{lemma_case1} we get that $I_1 \le |B_r|^\mu$ for all $r >0$ .

Next, we estimate $I_2$. Since for $x \in B_r$ and $|y| \geq 2 r$ we have $|x| < r \leq \frac{|y|}{2}$ and 
$| x - y| \geq |y| - |x| > \frac{|y|}{2}$ it follows that 
\begin{equation} \label{A}
|I_{\alpha}(f(1-\chi_{B_{2r}}))(x)| \leq \int_{{\mathbb R^n}\setminus B_{2r}} \dfrac {|f(y)|} {|x - y|^{n-\alpha}} \,dy 
\leq 2^{n-\alpha} \, \int_{{\mathbb R^n}\setminus B_{2r}} \dfrac {|f(y)|} {|y|^{n-\alpha}} \,dy. 
\end{equation}
By Lemma \ref{lemma3} and the estimate (\ref{Phi14}) we obtain
\begin{eqnarray*}
\Psi \left( \dfrac {\int_{{\mathbb R^n}\setminus B_{2r}} \dfrac {|f(y)|} {|y|^{n-\alpha}} \,dy}{C_3 \,C_4 \, \|f\|_{M^{\Phi,\lambda}(0)}}  \right) \,dx 
&\leq & 
\Psi \left( \frac{1}{C_4} \int_{|B_{2r}|} t^{\alpha/n} \Phi^{-1}(t^{\lambda -1}) \frac{dt}{t}\right) \\
&\leq &\
\Psi \left(\Psi^{-1}( |B_{2r}|^{\mu-1})\right) \leq |B_{2r}|^{\mu-1}. 
\end{eqnarray*}
Thus, for $x \in B_r$ 
\begin{equation*}
I_2 
\leq \int_{B_r} \Psi \left( \dfrac {\int_{{\mathbb R^n}\setminus B_{2r}} \dfrac {|f(y)|} {|y|^{n-\alpha}} \,dy}{C_3 \,
C_4 \, \|f\|_{M^{\Phi,\lambda}(0)}} \right) \,dx \leq |B_{2r}|^{\mu-1} \cdot |B_r| < |B_r|^{\mu}. 
\end{equation*}
Hence, 
\begin{equation*}
\dfrac 1 {|B_r|^{\mu}} \int_{B_r} \Psi \left(\dfrac {|I_{\alpha}f(x)|} {C_6 \, \|f\|_{M^{\Phi,\lambda}(0)}} \right)\,dx < 1,
\end{equation*}
and so
\begin{equation*}
\|I_{\alpha}f\|_{M^{\Psi,\mu}(0)} \leq C_6 \, \|f\|_{M^{\Phi,\lambda}(0)}. 
\end{equation*}

(ii) Similarly to the previous case, by (\ref{defined}), we obtain for $u > 0$
\begin{eqnarray*}
&& \Psi \left( \dfrac {|I_{\alpha}f(x)|}{c_{6} \, \|f\|_{M^{\Phi,\lambda}(0)}} \right) \\
&& \quad \ \ \ \leq \ 
\dfrac{1}{2}  \Psi \left( \dfrac {|I_{\alpha}(f\chi_{B_{2r}})(x)|} {c_7\,\|f\|_{M^{\Phi,\lambda}(0)}} \right) 
  + \dfrac{1}{2} \Psi \left( \dfrac {|I_{\alpha}(f(1-\chi_{B_{2r}}))(x)|} {2^{n-\alpha+1} \, C_3 \, C_4 \, \|f\|_{M^{\Phi,\lambda}(0)}} \right) \\
&& \quad \ \ \ =: \ \dfrac{1}{2} \left( I_3 + I_4 \right), 
\end{eqnarray*}
with $c_6 = 2 \,\max(c_7,\, 2^{n-\alpha+1}\,C_3\, C_4, c_7 = 2^{n(\mu-\lambda)+1} c_9$ and
 $c_9 = 2\, C_5 \max\left(\frac 2 {\ln 2} c_0 \,C_8, C_3 \right)$.
 
 Since $\Psi (u) \, d(g, u) = v \, d(g, \Psi^{-1}(v)) = v \, d(\Psi(g), v)$ for any $u > 0$ with $v = \Psi (u)$ and 
 $$
 d\left(\Psi\left(\frac{|I_\alpha f (x)|}{c_{6} \|f\|_{M^{\Phi, \lambda}(0)}}\right), u\right) 
 \leq d \left(I_3, u\right) + d \left(I_4, u\right),
 $$
it follows that
\begin{eqnarray*}
\sup\limits_{u>0} \frac {\Psi(u)}{{|B_r|^\mu}} \, d\left(\frac{|I_\alpha f (x)|}{c_{6}\|f\|_{M^{\Phi,\lambda}(0)}}, u\right) 
&\le& \sup\limits_{u>0} \frac u {|B_r|^\mu} \, d\left(I_3, u \right) + \sup\limits_{u>0} \frac u {|B_r|^\mu} \, d\left(I_4, u \right).
\end{eqnarray*}
From the proof of Lemma \ref{lemma_case1} for all $r >0$ 
$$
I_3 = \Psi \left( \dfrac {|I_{\alpha}(f\chi_{B_{2r}})(x)|} {c_7 \, \|f\|_{M^{\Phi,\lambda}(0)}} \right) 
\le \frac 1 2 |B_r|^{\mu - \lambda} \,\Phi\left(\frac{Mf(x)}{\|Mf\|_{WM^{\Phi,\lambda}(0)}}\right)
$$
and
\begin{eqnarray*}
\sup\limits_{u>0} \frac u {|B_r|^\mu} \, d\left(I_1, u \right) &\le& \sup\limits_{u>0} \frac u {|B_r|^\mu} \,
d \left(\frac 1 2 |B_r|^{\mu-\lambda} \Phi\left(\frac{Mf(x)}{\|Mf\|_{WM^{\Phi,\lambda}(0)}}\right), u \right)\\
&=& \frac 1 2 \sup\limits_{u>0} \frac u {|B_r|^\lambda} \, d \left(\Phi\left(\frac{Mf(x)}{\|Mf\|_{WM^{\Phi,\lambda}(0)}}\right), u \right)\\
&=& \frac 1 2 \sup\limits_{u>0} \frac {\Phi(u)}{{|B_r|^\lambda}}\, d \left(\frac{Mf(x)}{\|Mf\|_{WM^{\Phi,\lambda}(0)}}, u \right) \le \frac 1 2.
\end{eqnarray*}
For $I_4$, using Lemma \ref{lemma3} we obtain
$$
I_4 = \Psi \left( \dfrac {|I_{\alpha}(f(1-\chi_{B_{2r}}))(x)|} {2^{n-\alpha+1} \, C_3 \, C_4 \, \|f\|_{M^{\Phi,\lambda}(0)}} \right) 
\le \frac{1}{2} |B_r|^{\mu-1}
$$
and
$$
\sup\limits_{u>0} \frac u {|B_r|^\mu} \, d\left(I_4, u \right) \le \sup\limits_{u>0} \frac u {|B_r|^\mu} \,
d \left(\frac 1 2 |B_r|^{\mu-1}, u \right) = \frac{1}{2} \sup\limits_{u>0} u~ d\left(\frac 1 {|B_r|},u\right) \le \frac{1}{2}.
$$
Thus,
$$
\sup\limits_{u>0} \frac{\Psi(u)} {{|B_r|^\mu}} \, d\left(\frac{|I_\alpha f (x)|}{c_6 \|f\|_{M^{\Phi,\lambda}(0)}}, u\right) \le 1
$$
and $\|I_{\alpha}f\|_{WM^{\Psi,\mu}(0)} \leq c_{6} \, \|f\|_{M^{\Phi,\lambda}(0)}.$ 
\end{proof}

\begin{example}
Let $0<\alpha<n, 1 < p < \frac{n(1-\lambda)}{\alpha}, 0 \leq \lambda < 1,$ and 
$$
\Phi(u) = u^p, \quad \Psi(u) = u^q \quad \text{with} \quad 1 < p < q < \infty.
$$ 
Then ${\Phi}^*(u) = (p - 1) \, p^{- p^{\prime}} u^{p^{\prime}}$, where $1/p + 1/p^{\prime} = 1$ and ${\Phi}^*(2u) = 2^{p^{\prime}} {\Phi}^*(u)$, 
that is, ${\Phi}^* \in {\Delta}_2$. The estimate (\ref{Phi14}) holds since 
$$
\int_u^{\infty} t^{\alpha/n} \Phi^{-1}(t^{\lambda-1}) \, \frac{dt}{t} = \int_u^{\infty} t^{\frac{\alpha}{n} + \frac{\lambda-1}{p}} \, \frac{dt}{t} 
= \frac{1}{\frac{1 - \lambda}{p} - \frac{\alpha}{n}} \, u^{\frac{\alpha}{n} + \frac{\lambda - 1}{p}}
$$
for all $u > 0$, where the last integral is convergent because $ p < \frac{n(1-\lambda)}{\alpha}$. If $ \frac{1}{q} = \frac{1}{p} - \frac{\alpha}{n}$ 
and $\frac{\lambda}{p} = \frac{\mu}{q}$, then $\frac{\alpha}{n} + \frac{\lambda - 1}{p} = \frac{\lambda}{p} - (\frac{1}{p} - \frac{\alpha}{n}) 
= \frac{\mu}{q} - \frac{1}{q}$ and 
$$
\int_u^{\infty} t^{\alpha/n} \Phi^{-1}(t^{\lambda-1}) \, \frac{dt}{t} = \frac{q}{1 - \mu} \, u^{\frac{\mu - 1}{q}} = \frac{q}{1 - \mu} \, \Psi^{-1}(u^{\mu-1}), 
$$ 
that is, the estimate (\ref{Phi14}) holds. Also estimate (\ref{Orlicz15}) holds since for all $u, r > 0$
$$
\int_u^{\infty} t^{\frac{\alpha}{n}} \Phi^{-1}(\frac{r^{\lambda}}{t}) \, \frac{dt}{t} = r^{\frac{\lambda}{p}} 
 \int_u^{\infty} t^{\frac{\alpha}{n} - \frac{1}{p}} \, \frac{dt}{t} = \frac{ r^{\frac{\lambda}{p}}}{\frac{1}{p} - \frac{\alpha}{n} } 
 u^{\frac{\alpha}{n} - \frac{1}{p}} = q \, r^{\frac{\mu}{q}} u^{-1/q} = q \, \Psi^{-1}(\frac{r^{\mu}}{u}).
$$
From the Theorem \ref{Thm3} we get the Spanne--Peetre type result proved in \cite[Proposition 1.1]{FLL08}, that is, 
the Riesz potential $I_\alpha$ is bounded from $M^{p,\lambda}(0)$ to $M^{q,\mu}(0)$ under the conditions 
$1 < p < \frac{n(1-\lambda)}{\alpha}, 0 \leq \lambda < 1, \frac{1}{q} = \frac 1 p - \frac{\alpha}{n}$ and $\frac{\lambda}{p} 
= \frac{\mu}{q}$.
\end{example}

\begin{remark} \label{rem3}
It is easy to see that for $0 \leq \lambda < 1$ if $\Phi_1, \Phi_2$ are two Orlicz functions and there exists a constant 
$k > 0$ such that $\Phi_2(u) \leq \Phi_1(k u)$ for all $u > 0$, then $\| f\|_{\Phi_2, \lambda,A } \leq k \, \| f\|_{\Phi_1, \lambda,A } $ 
provided the right side is finite. Furthemore, 
$ M^{\Phi_1, \lambda}({\mathbb R^n}) \overset{k}{\hookrightarrow} M^{\Phi_2, \lambda}({\mathbb R^n})$ and 
$M^{\Phi_1, \lambda}(0) \overset{k}{\hookrightarrow} M^{\Phi_2, \lambda}(0)$. Hence it follows that if two 
Orlicz functions $\Phi_1, \Phi_2$ are equivalent, i.e. there exist positive constants $k_1,k_2$ such that 
$\Phi_1(k_1 u) \leq \Phi_2(u) \leq \Phi_1(k_2 u)$ for all $u > 0$, then $ M^{\Phi_1, \lambda}({\mathbb R^n}) = 
M^{\Phi_2, \lambda}({\mathbb R^n})$ and $M^{\Phi_1, \lambda}(0) = M^{\Phi_2, \lambda}(0)$ with equivalent 
norms.
\end{remark}

\begin{example} \label{ex2}
Let $0 < \alpha < n, 0 \leq \lambda < 1, 1 < p < \frac{n(1-\lambda)}{\alpha}, a > 0$ and 
\begin{equation*} 
\Phi^{-1}(u) = 
\begin{cases}
u^{\frac 1 p} & \text{for}~0 \leq u \leq 1,\\
u^{\frac 1 p} \left(1 + \ln u \right)^{-a} & \text{for}~u \geq1,
\end{cases}
\quad \Psi^{-1}(u) =  u^{\frac{1}{q}} ~~ \text{with} ~ 1 < p < q < \infty.
\end{equation*}
If $\frac{1}{q} = \frac{1}{p} - \frac{\alpha}{n}, \frac{\lambda}{p} = \frac{\mu}{q}$, then condition 
(\ref{Phi14}) is satisfied. Really, for $u \geq 1$ we have equality as in the Example 1. If $0 < u < 1$, then 
using the fact that function $ \left(1 + \ln t^{\lambda - 1} \right)^{-a} $ is strictly increasing of variable $t$ on $(0, 1]$,  
we get $ \left(1 + \ln t^{\lambda - 1} \right)^{-a} \leq 1$ for $0 < t \leq 1$ and so
\begin{eqnarray*}
 \int_u^{\infty} t^{\alpha/n} \Phi^{-1}(t^{\lambda-1}) \, \frac{dt}{t} &=& \int_u^{1} t^{\frac{\alpha}{n} + \frac{\lambda-1}{p}} 
 (1 + \ln t^{\lambda-1})^{-a} \, \frac{dt}{t} +  \int_1^{\infty} t^{\frac{\alpha}{n} + \frac{\lambda-1}{p}} \, \frac{dt}{t} \\
&\leq&  \int_u^{\infty}  t^{\frac{\alpha}{n} + \frac{\lambda-1}{p}} \, \frac{dt}{t} = 
\frac{u^{\frac{\alpha}{n} + \frac{\lambda - 1}{p}}}{\frac{1 - \lambda}{p} - \frac{\alpha}{n}} = \frac{q}{1 - \mu} \, u^{\frac{\mu - 1}{q}} 
= \frac{q}{1 - \mu} \, \Psi^{-1}(u^{\mu-1}),
\end{eqnarray*}
that is, the estimate (\ref{Phi14}) holds. Next, we consider condition (\ref{Orlicz15}). If $u \geq r^{\lambda}$, then 
$$
\int_u^{\infty} t^{\frac{\alpha}{n}} \Phi^{-1}(\frac{r^{\lambda}}{t}) \, \frac{dt}{t} = r^{\frac{\lambda}{p}} 
 \int_u^{\infty} t^{\frac{\alpha}{n} - \frac{1}{p}} \, \frac{dt}{t} = \frac{ r^{\frac{\lambda}{p}}}{\frac{1}{p} - \frac{\alpha}{n} } 
\, u^{\frac{\alpha}{n} - \frac{1}{p}} = q \, r^{\frac{\mu}{q}} u^{- \frac{1}{q}} = q \, 
\Psi^{-1}(\frac{r^{\mu}}{u}).
$$
Let now $0 < u < r^{\lambda}$. Then, $(1 + \ln\frac{r^{\lambda}}{t})^{-a} \leq 1$ as an increasing function of 
$t$ on $(0, r^{\lambda}]$ and since $u < t \leq r^{\lambda}$, we have
\begin{eqnarray*}
\int_u^{\infty} t^{\frac{\alpha}{n}} \Phi^{-1}(\frac{r^{\lambda}}{t}) \, \frac{dt}{t} &=& r^{\frac{\lambda}{p}} 
 \int_u^{r^{\lambda}} t^{\frac{\alpha}{n} - \frac{1}{p}} (1 + \ln \frac{r^{\lambda}}{t})^{-a}\, \frac{dt}{t} + 
 r^{\frac{\lambda}{p}} \int_{r^{\lambda}}^{\infty} t^{\frac{\alpha}{n} - \frac{1}{p}} \, \frac{dt}{t}\\
 &\leq& r^{\frac{\lambda}{p}} \int_u^{r^{\lambda}} t^{\frac{\alpha}{n} - \frac{1}{p}} \, \frac{dt}{t} +  
 r^{\frac{\lambda}{p}} \int_{r^{\lambda}}^{\infty} t^{\frac{\alpha}{n} - \frac{1}{p}} \, \frac{dt}{t} \\
 &=& r^{\frac{\lambda}{p}} \int_u^{\infty} t^{\frac{\alpha}{n} - \frac{1}{p} } \, \frac{dt}{t} = 
\frac{ r^{\frac{\lambda}{p}}}{\frac{1}{p} - \frac{\alpha}{n} } \, u^{\frac{\alpha}{n} - \frac{1}{p}} 
= q \, r^{\frac{\mu}{q}} u^{- \frac{1}{q}} = q \, \Psi^{-1}(\frac{r^{\mu}}{u}),
\end{eqnarray*}
that is, the estimate (\ref{Orlicz15}) holds. The function $\Phi^{-1}$ is increasing, unbounded, obviously 
concave on $(0, 1)$ and concave for large $u$. Therefore, there exists a concave function on $(0, \infty)$ 
which is equivalent to $\Phi^{-1}$ and so $\Phi$ is equivalent to an Orlicz function. Also we have equivalence
$$
\Phi (u) \approx 
\begin{cases}
u^p & \text{for}~0 \le u \le 1,\\
u^p \left(1 + \ln u \right)^{a p} & \text{for}~u \ge 1.
\end{cases}
$$
Moreover, since
 $$
s_{\Phi^{-1}}(t) = 
\begin{cases}
 t^{1/p} (1 - \ln t)^{a} & \text{for}~0 < t \le 1,\\
 t^{1/p}  & \text{for}~t\ge 1,
\end{cases}
$$
it follows that the Matuszewska--Orlicz index $\beta_{\Phi^{- 1}} = \frac{1}{p}$ and so 
$1 = \frac{1}{\beta_{\Phi^*}} + \frac{1}{\alpha_{\Phi}} = \frac{1}{\beta_{\Phi^*}} +\beta_{\Phi^{- 1}} 
= \frac{1}{\beta_{\Phi^*}} + \frac{1}{p}$ or $\beta_{\Phi^*} = \frac{p}{p - 1} < \infty$, which means that 
$\Phi^* \in \Delta_2$ (for definitions and properties of indices -- see \cite[pp. 87--89]{Ma89}). Thus, 
by Remark \ref{rem3}, the space $M^{\Phi, \lambda}(0)$ is a Banach space and by Theorem \ref{Thm3} 
the Riesz potential $I_\alpha$ is bounded from $M^{\Phi,\lambda}(0)$ to $M^{\Psi,\mu}(0) = M^{q, \mu}(0)$.
\end{example}

\begin{example} \label{ex3}
Let $0 < \alpha < n, 0 \leq \lambda < 1, 1 < p < \frac{n(1-\lambda)}{\alpha}, 0 \leq b \leq a$ and 
\begin{equation*} 
\Phi^{-1}(u) = u^{\frac 1 p} \left(1 + |\ln u| \right)^{-a} ~~ \text{and} ~~ \Psi^{-1}(u) = u^{\frac 1 q} \left(1 + |\ln u| \right)^b 
~~ \text{for} ~~ u >0.
\end{equation*}
If $\frac{1}{q} = \frac{1}{p} - \frac{\alpha}{n}, \frac{\lambda}{p} = \frac{\mu}{q}$, then conditions (a), (b) of 
Theorem \ref{necessity}(i) and (\ref{Phi14}), (\ref{Orlicz15}) are satisfied. 
The calculations are similar to those in Example \ref{ex2} so we will omit them here. Observe only that
 $$
s_{\Phi^{-1}}(t) = t^{1/p} (1 + |\ln t|)^a, \, s_{\Psi^{-1}}(t) = t^{1/q} (1 + |\ln t|)^b.
$$
Then, the functions $\Phi^{-1}, \Psi^{-1}$ are increasing, unbounded and concave near $0$ and for large $u$, 
and so the inverses $\Phi, \Psi$ are equivalent to Orlicz functions. Thus, by Remark \ref{rem3}, the spaces 
$M^{\Phi, \lambda}(0), M^{\Psi, \mu}(0)$ are Banach spaces and by Theorem \ref{Thm3} 
the Riesz potential $I_\alpha$ is bounded from $M^{\Phi,\lambda}(0)$ to $M^{\Psi,\mu}(0)$.
\end{example}

\begin{acknowledgement}
\rm{The third author was supported by Grant-in-Aid for Scientific Research (C), No. 17K05306, No. 20K03663, 
Japan Society for the Promotion of Science. }
\end{acknowledgement}


\vspace{3mm}

\noindent
{\footnotesize Department of Engineering Sciences and Mathematics\\
Lule\r{a} University of Technology, SE-971 87 Lule\r{a}, Sweden\\
~{\it E-mail address: {\tt evgeniya.burtseva@ltu.se}}\\
}

\vspace{-2mm}

\noindent
{\footnotesize Department of Engineering Sciences and Mathematics\\
Lule\r{a} University of Technology, SE-971 87 Lule\r{a}, Sweden\\
~{\it E-mail address: {\tt lech.maligranda@ltu.se}}\\
and\\
{\footnotesize Institute of Mathematics, Pozna\'n University of Technology\\
ul. Piotrowo 3a, 60-965 Pozna\'n, Poland\\
~{\it E-mail address: {\tt lech.maligranda@put.poznan.pl}}\\
}}

\vspace{-2mm}

\noindent
{\footnotesize College of Economics, Nihon University\\
1-3-2 Kanda, Misaki-cho, Chiyoda-ku, Tokyo 101-8360, Japan\\ 
{\it E-mail address: {\tt katsu.m@nihon-u.ac.jp}} \\
}

\end{document}